\documentclass[a4paper]{amsart}
\usepackage{amssymb}
\usepackage{amsmath}
\usepackage{xypic}

\newtheorem{thm}{Theorem}[section]
\newtheorem*{thm*}{Theorem}

\newtheorem{cor}[thm]{Corollary}
\newtheorem*{cor*}{Corollary}
\newtheorem{lemma}[thm]{Lemma}
\newtheorem{prop}[thm]{Proposition}
\theoremstyle{definition}
\newtheorem{defn}[thm]{Definition}
\newtheorem{ex}[thm]{Example}

\newtheorem*{quest*}{Question}

\theoremstyle{remark}
\newtheorem{rem}[thm]{Remark}
\newtheorem*{rem*}{Remark}

\newtheorem*{org}{Organization of the article}
\newtheorem*{ack}{Acknowledgment}

\newdimen\theight
\def\TeXref#1{%
             \leavevmode\vadjust{\setbox0=\hbox{{\tt
                     \quad\quad  {\small \textrm #1}}}%
             \theight=\ht0
             \advance\theight by \lineskip
             \kern -\theight \vbox to
             \theight{\rightline{\rlap{\box0}}%
             \vss}%
             }}%

\newcommand{\ZZ}{{\mathbb  Z}}
\newcommand{\RR}{{\mathbb  R}}
\newcommand{\CC}{{\mathbb  C}}

\renewcommand{\cD}{{\mathcal D}}

\newcommand{\E}{{\mathcal E}}

\newcommand{\F}{{\mathcal F}}

\newcommand{\mfg}{\mathfrak{g}}

\newcommand{\mfh}{\mathfrak{h}}

\newcommand{\mfa}{\mathfrak{a}}

\newcommand{\mft}{\mathfrak{t}}

\newcommand{\rk}{\operatorname{rank}}

\newcommand{\id}{\operatorname{id}}
\newcommand{\odd}{\operatorname{odd}}
\newcommand{\rank}{\operatorname{Rank}}

\newcommand{\SO}{\operatorname{SO}}

\newcommand{\U}{\operatorname{U}}
\renewcommand{\S}{\operatorname{S}}

\newcommand{\Isom}{\operatorname{Isom}}

\newcommand{\End}{\operatorname{End}}
\newcommand{\delbar}{\operatorname{\overline{\partial}}}

\newcommand{\trX}{\operatorname{\widetilde{X}}}

\newcommand{\SU}{\operatorname{SU}}
\newcommand{\Image}{\operatorname{Image}}

\newcommand{\stb}{\star_b}
\newcommand{\bdel}{\overline{\partial}}
\newcommand{\AUT}{\mathfrak{Aut}}
\newcommand{\CR}{\mathfrak{CR}}
\newcommand{\Lcr}{\mathfrak{cr}}

\numberwithin{equation}{section}

\begin{document}

\title[Basic Dolbeault cohomology of Sasakian manifolds]{Rigidity and vanishing of basic Dolbeault cohomology of Sasakian manifolds}

\author{Oliver Goertsches}
\address{Oliver Goertsches, Fachbereich Mathematik, Universit\"at Hamburg, Bundesstra\ss e 55, 20146 Hamburg, Germany}
\email{oliver.goertsches@math.uni-hamburg.de}
\author{Hiraku Nozawa}
\thanks{The second author is partially supported by Research Fellowship of Canon Foundation in Europe, the EPDI/JSPS/IH\'{E}S Fellowship and the Spanish MICINN grants MTM2011-25656. This paper was written during the stay of the second author at Institut des Hautes \'{E}tudes Scientifiques (Bures-sur-Yvette, France), Institut Mittag-Leffler (Djursholm, Sweden) and Centre de Recerca Matem\`{a}tica (Bellaterra, Spain); he is very grateful for their hospitality.}
\address{Hiraku Nozawa, Centre de Recerca Matem\`{a}tica, Campus de Bellaterra, Edifici C, 08193 Bellaterra, Barcelona, Spain}
\email{nozawahiraku@06.alumni.u-tokyo.ac.jp}
\author{Dirk T\"oben}
\thanks{The third author is supported by a fellowship of CNPq-Brazil}
\address{Dirk T\"oben, Instituto de Matem\'atica e Estat\'istica,
 Universidade de S\~ao Paulo,  Rua do Mat\~ao, 1010
 S\~ao Paulo, SP 05508-090, Brazil}
\email{dtoeben@googlemail.com}

\keywords{Sasakian manifolds, Killing foliations, basic cohomology, basic Dolbeault cohomology, transverse Hodge theory, equivariant cohomology, equivariant formality}
\date{}

\subjclass[2010]{53D35, 55N25, 58A14}

\begin{abstract}
The basic Dolbeault cohomology $H^{p,q}(M,\F)$ of a Sasakian manifold $(M,\eta,g)$ is an invariant of its characteristic foliation $\F$ (the orbit foliation of the Reeb flow). We show some fundamental properties of this cohomology, which are useful for its computation. In the first part of the article, we show that the basic Hodge numbers $h^{p,q}(M,\F)$, the dimensions of $H^{p,q}(M,\F)$, only depend on the isomorphism class of the underlying CR structure. Equivalently, we show that the basic Hodge numbers are invariant under deformations of type I. This result reduces the computation of $h^{p,q}(M,\F)$ to the quasi-regular case. In the second part, we show a basic version of the Carrell-Lieberman theorem relating $H^{\bullet,\bullet}(M,\F)$ to $H^{\bullet,\bullet}(C,\F)$, where $C$ is the union of closed leaves of $\F$. As a special case, if $\F$ has only finitely many closed leaves, then we get $h^{p,q}(M,\F)=0$ for $p\neq q$. Combining the two results, we show that if $M$ admits a nowhere vanishing CR vector field with finitely many closed orbits, then $h^{p,q}(M,\F)=0$ for $p\neq q$. As an application of these results, we compute $h^{p,q}(M,\F)$ for deformations of homogeneous Sasakian manifolds.
\end{abstract}

\maketitle

\tableofcontents

\addtocontents{toc}{\protect\setcounter{tocdepth}{1}}
\section{Introduction} \label{sec-intro}

\subsection{Background}
Recently Sasakian manifolds have been studied by Einstein geometers and physicists as they provide examples of odd dimensional Einstein manifolds which appear in the AdS/CFT correspondence~\cite{Boyer Galicki 4,GMSW,BGK,MSY,van Coevering,FOW}. We refer to~\cite[Section~14.5]{Boyer Galicki} for a brief historical account of the physical background for mathematicians and a list of references in physics. The basic Dolbeault cohomology of the characteristic foliation of a Sasakian manifold is a fundamental invariant similar to the Dolbeault cohomology of a K\"{a}hler manifold. This cohomology has good properties; for example, El Kacimi-Alaoui~\cite{ElKacimi} proved basic versions of the Hodge and Lefschetz decompositions. In the quasi-regular case, where all leaves of the characteristic foliation are circles, the computation of this cohomology reduces to that of complex orbifolds, but only little is known about how to calculate it in the irregular case. One of the few examples is a Kodaira-Akizuki-Nakano-type vanishing theorem~\cite[Theorem~1.2]{Nozawa} which is valid for positive Sasakian manifolds.

In this article we will show fundamental properties of the basic Dolbeault cohomology described in the abstract, which are useful to compute the basic Hodge numbers of irregular Sasakian manifolds. We will now explain our results in more detail.

\subsection{Rigidity of the basic Hodge numbers of Sasakian manifolds}\label{sec:ridintro}
\subsubsection{A rigidity theorem}
Surprisingly, the basic Betti numbers of the characteristic foliation, i.e., the dimensions of the basic cohomology groups, are the same for any Sasakian structure on a fixed compact manifold~\cite[Theorem~7.4.14]{Boyer Galicki}, which means that they cannot distinguish Sasakian structures. In contrast, the basic Hodge numbers can, see Example~\ref{ex:basicDolbeaultdistinguishes} below. Therefore the question arises on which qualities of a Sasakian structure the latter groups depend. Our first result is the following:
\begin{thm*}
Two Sasakian structures on a compact manifold with isomorphic CR structures have the same basic Hodge numbers.
\end{thm*}
As we explain in Section~\ref{sec:cons}, this allows to reduce the computation of basic Hodge numbers to the quasi-regular case.
\begin{rem*}
Here two CR structures $({\mathcal D}_{1},J_{1})$ and $({\mathcal D}_{2},J_{2})$ on a manifold $M$ are called isomorphic if there exists a diffeomorphism of $M$ that maps $({\mathcal D}_{1},J_{1})$ to $({\mathcal D}_{2},J_{2})$. Note that the statement is equivalent to the equality of the basic Hodge numbers of two Sasakian structures with the same underlying CR structure.
\end{rem*}

The above theorem can be also interpreted as the rigidity of the basic Hodge numbers under certain deformations of Sasakian manifolds: It is well-known that a Sasakian structure on a manifold $M$ is determined by the underlying CR structure and the Reeb vector field. Denoting by ${\mathcal S}(\mathcal D,J)$ the space of Sasakian structures on a compact manifold $M$ with a fixed underlying CR structure $({\mathcal D},J)$, we therefore see that an element of ${\mathcal S}(\mathcal D,J)$ can be identified with its Reeb vector field. As proven in~\cite{BGS}, see Section~\ref{sec:cone} below, this gives an isomorphism between ${\mathcal S}(\mathcal D,J)$ and an open cone in the Lie algebra of the CR diffeomorphism group of $({\mathcal D},J)$. A {\em deformation of type I}~\cite{GauOr,Belgun} (see also~\cite[Section~8.2.3]{Boyer Galicki}) is a deformation of Sasakian structures inside ${\mathcal S}(\mathcal D,J)$. Since any two Sasakian structures in ${\mathcal S}(\mathcal D,J)$ can be connected by a deformation of type I, the above theorem is equivalent to the following rigidity statement.
\begin{thm*}
The basic Hodge numbers of a Sasakian structure on a compact manifold are invariant under deformations of type I.
\end{thm*}

\begin{rem*}
The basic Hodge numbers are also invariant under deformations of type II (\cite[Definition~7.5.9]{Boyer Galicki}), because under such deformations the characteristic foliation and the transverse holomorphic structure do not change.
\end{rem*}

\begin{rem*}
We actually prove a stronger result, namely Theorem~\ref{thm:Invariance2}: the basic Hodge numbers are constant for a smooth family of Sasakian structures, if there exists a smooth family of actions of a compact Lie group, such that for each parameter value it preserves the CR structure and contains the Reeb flow.
\end{rem*}

To prove the rigidity theorem, we want to apply Kodaira-Spencer theory to the family of basic Laplacians. Unfortunately, we cannot use this theory directly, because the basic de Rham complex changes discontinuously. To avoid this difficulty, we consider transverse forms, which are sections of exterior products of the complexified conormal bundles of the characteristic foliations. Extending the basic Laplacians to strongly elliptic operators acting on them in a way motivated by a construction of El Kacimi-Alaoui and Hector~\cite{ElKacimiHector,ElKacimi}, Kodaira-Spencer theory can be applied to show the constancy of the basic Hodge numbers; here an additional argument is needed to go back from the level of transverse forms to that of basic forms.

We wonder if the rigidity holds more generally:
\begin{quest*}
Are the basic Hodge numbers of Sasakian manifolds invariant under general smooth deformations ?
\end{quest*}
For a description of general deformations of characteristic foliations, we refer to~\cite[Section~8.2]{Boyer Galicki} and references therein. A family of transversely Hermitian foliations whose basic Hodge numbers change can be obtained by taking the product with $S^{1}$ of a family of Riemannian foliations in~\cite[Example~7.4]{Nozawa 2}. But we do not know if the basic Hodge numbers of transversely K\"{a}hler foliations can change. K\"{a}hler manifolds with holomorphic vector fields have certain rigidity as shown in~\cite{AMN}, which makes the construction of counterexamples difficult.

\subsubsection{Corollaries of the rigidity theorem}\label{sec:cons}
The rigidity result in the last section is useful for computations. It is well-known that the subset of quasi-regular Sasakian structures (i.e., whose Reeb orbits are closed) is dense in ${\mathcal S}(\mathcal D,J)$. Note that the basic Dolbeault cohomology of a quasi-regular Sasakian structure is the Dolbeault cohomology of its leaf space as a direct consequence of the definition. Since the leaf space is a complex projective orbifold, we get the following corollary:
\begin{cor*}
The basic Hodge numbers of a compact Sasakian manifold are the Hodge numbers of a complex projective orbifold.
\end{cor*}

The positivity of Sasakian structures (see Definition~\ref{def:pos}) is preserved under small deformations of type I, because it is an open condition. By the invariance of basic Hodge numbers and the Kodaira-Akizuki-Nakano-type vanishing theorem for quasi-regular Sasakian manifolds~\cite[Proposition~2.4]{Boyer Galicki Nakamaye2}, we get another proof of \cite[Theorem~1.2]{Nozawa}:
\begin{cor*}
The $(p,0)$-th and $(0,q)$-th basic Hodge numbers of a positive compact Sasakian manifold are zero for $p>0$ and $q>0$.
\end{cor*}

\subsubsection{Deformation of basic Hodge decompositions}
Our proof of the rigidity theorem in Section~\ref{sec:ridintro} also gives the following result:
\begin{thm*}
Within the space ${\mathcal S}(\mathcal D,J)$ of Sasakian structures with the same CR structure, the Hodge decomposition $H^k(M,\F)=\bigoplus_{p+q=k}H^{p,q}(M,\F)$ smoothly depends on the Reeb vector field defining the foliation.
\end{thm*}
More precisely, this means that the family of harmonic spaces and its decomposition $\mathcal H^k(M,\F)=\bigoplus_{p+q=k}\mathcal H^{p,q}(M,\F)$ is differentiable in $\Omega^\bullet(M)$.

\begin{rem*}
Note that the basic cohomology of the characteristic foliation does not admit a Hodge structure in the classical sense, because it does not have any natural integer lattice. One may consider the embedding $H^k(M,\F) \to H^{k}(M;\RR)$, but the intersection $H^k(M,\F) \cap H^{k}(M;\ZZ)$ may be trivial.
\end{rem*}

\subsection{A Carrell-Lieberman-type vanishing theorem}\label{sec:CLintro}

\subsubsection{Localization of the basic cohomology}
For a Sasakian manifold $(M,\eta,g)$, or more generally a $K$-contact manifold, with characteristic foliation $\F$, the basic cohomology of $(M,\F)$ is localized at the union $C$ of closed leaves of $\F$~\cite[Theorem~7.11]{GNT}; namely, $\sum_k b^k(M,\F)=\sum_k b^k(C,\F)$, where $b^{k}$ denotes the basic Betti numbers. This is deduced from a Borel localization theorem~\cite{GT2010} for more general Riemannian foliations, called Killing foliations (see Section~\ref{sec:traction} for the definition). This gives us another strategy to compute the basic cohomology, namely by deforming the Sasakian structure so that the union of closed leaves of the characteristic foliation is as simple as possible. In the second part of the article, we refine this argument to basic Dolbeault cohomology to prove a basic version of the classical Carrell-Lieberman theorem~\cite{CarrellLieberman} for Lie group actions on K\"{a}hler manifolds.

\subsubsection{Statement of the theorem}\label{sec:stCL}
We will show a Carrell-Lieberman-type theorem for general transversely K\"ahler Killing foliations on compact manifolds such that the transverse action of their Molino sheaf is equivariantly formal, see Theorem~\ref{thm:carlieb}. Since the characteristic foliation of a Sasakian manifold is equivariantly formal by~\cite[Theorem~6.8]{GNT}, we get the following:
\begin{thm*}
For a Sasakian structure on a compact manifold $M$ with characteristic foliation $\F$, let $C$ be the union of the closed leaves of $\F$ and $C/\F$ the leaf space of $(C,\F)$, which admits naturally the structure of a K\"{a}hler orbifold. We have
\[
\sum_p h^{p,p+s}(M,\F)=\sum_p h^{p,p+s}(C,\F)=\sum_p h^{p,p+s}(C/\F)
\]
for all $s$. In particular, $h^{p,p+s}(M,\F)=0$ for $|s|>\dim_\CC C/\F$.
\end{thm*}

This theorem is proved by an adaption of a new proof of the Carrell-Lieberman theorem due to Carrell, Kaveh and Puppe~\cite{CKP}, based on equivariant Dolbeault cohomology, to the basic setting by introducing a notion of equivariant basic Dolbeault cohomology.

\subsubsection{Corollaries of the Carrell-Lieberman-type theorem}\label{sec:corCL}
The Carrell-Lieberman-type theorem in the last section implies the following vanishing statement:
\begin{thm*}
If the characteristic foliation $\F$ of a Sasakian structure on a compact manifold $M$ has only finitely many closed leaves, then $h^{p,q}(M,\F)=0$ for $p\neq q$.
\end{thm*}

A vector field $X$ on a CR manifold is called {\em CR} if the flow generated by $X$ preserves the CR structure (note that some authors use this terminology in a different sense). In combination with the rigidity theorem in Section~\ref{sec:ridintro} we obtain:
\begin{thm*}
Let $(M,\eta,g)$ be a compact Sasakian manifold with characteristic foliation $\F$. If there exists a nowhere vanishing CR vector field on $M$ with finitely many closed orbits, then $h^{p,q}(M,\F)=0$ for $p\neq q$.
\end{thm*}

\subsection{Deformations of homogeneous Sasakian manifolds}\label{sec:defhom}
We illustrate the main theorems in this article with the example of homogeneous Sasakian manifolds. Note that it is well-known that any homogeneous Sasakian manifold is regular, and the total space of a circle bundle over a generalized flag manifold, which naturally admits the structure of K\"{a}hler manifold.
\begin{thm*}
If a compact manifold $M$ admits a homogeneous Sasakian structure, then $M$ admits also an irregular Sasakian structure whose characteristic foliation $\F$ has only finitely many closed leaves and
$$h^{p,q}(M,\F) = h^{p,q}(G/H)=
\begin{cases}
b^{2k}(G/H) & \text{if}\,\,\,\, p=q=k \;, \\
0 & \text{if}\,\,\,\, p \neq q\;,
\end{cases}
$$
where $G/H$ is the corresponding generalized flag manifold. The number of closed leaves of $\F$ is $\chi(G/H)$, the Euler number of $G/H$.
\end{thm*}

\begin{org}
Sections~\ref{sec:Sasakian},~\ref{sec:basiccoh} and \ref{sec:transverseactionsonchar} are devoted to recall fundamental notions as indicated in the table of contents. Results in Section~\ref{sec:ridintro} are proved in Section~\ref{sec:invarianceofhpq} (see Section~\ref{sec:st}). Equivariant basic Dolbeault cohomology is introduced in Section~\ref{sec:eqcohom} after recalling other cohomologies. The theorems in Sections~\ref{sec:stCL} and~\ref{sec:corCL} are deduced in Section~\ref{sec:vanSas} from the results in Section~\ref{sec:eqHodge}. The theorem in Section~\ref{sec:defhom} is proved in Section~\ref{sec:G/P}.
\end{org}

\begin{ack}
The authors are grateful to Marcel Nicolau for helpful discussions on K\"{a}hler manifolds and tranversely K\"{a}hler foliations. The authors are grateful to Charles Boyer for explaining to us Example~\ref{ex:basicDolbeaultdistinguishes} of two Sasakian structures on a manifold with different basic Hodge numbers.
\end{ack}

\section{Sasakian manifolds}\label{sec:Sasakian}

\subsection{Sasakian manifolds}
Let $M$ be an odd-dimensional manifold with a $1$-form $\eta$ and a Riemannian metric $g$.
\begin{defn}
$(M,\eta,g)$ is called a {\em Sasakian manifold} if $M \times \RR_{>0}$ is a K\"{a}hler manifold with metric $r^{2}g + dr \otimes dr$ and K\"{a}hler form $d(r^{2}\eta)$ where $r$ is the standard coordinate of $\RR_{>0}$.
\end{defn}

The {\em Reeb vector field} $\xi$ of $\eta$ is the vector field on $M$ defined by the equations $\eta(\xi)=1$ and $\iota_\xi d\eta=0$. A Sasakian manifold has an integrable CR structure $(\ker \eta,\Phi)$, where $\Phi$ is obtained from the complex structure on $M \times \RR_{>0}$.

\begin{rem}
In the book by Boyer-Galicki~\cite{Boyer Galicki}, the Sasakian structure is denoted by $(\eta,\xi,\Phi,g)$ including $\xi$ and $\Phi$ following preceding references. We will omit $\xi$ and $\Phi$ from this notation, because these are determined by $\eta$ and $g$.
\end{rem}

\begin{ex}
The odd-dimensional spheres with round metric and standard contact form are Sasakian. Other examples of Sasakian manifolds are the total spaces of circle bundles over K\"ahler manifolds whose Euler class is the K\"ahler class, contact toric manifolds of Reeb type (see~\cite[Theorem~5.2]{Boyer Galicki 2}), and links of isolated singularities of hypersurfaces defined by weighted homogeneous polynomials (\cite[Chapters~7~and~9]{Boyer Galicki} and references therein).
\end{ex}

\begin{rem}
Sasakian manifolds are examples of $K$-contact manifolds, i.e., contact manifolds whose Reeb flow preserve an (adapted) Riemannian metric. Hence, the results in~\cite{GNT} are applicable in our situation.
\end{rem}

\subsection{The Reeb flow and the characteristic foliation}\label{sec:torus}
The {\em Reeb flow} of $\eta$ is the flow generated by the Reeb vector field $\xi$ of $\eta$.
\begin{defn}
The orbit foliation $\F$ of the Reeb flow of $\eta$ is called the {\em characteristic foliation} of $(M,\eta,g)$.
\end{defn}
One can see that the Reeb flow of $\eta$ leaves $\eta$ and $g$ invariant by definition. If $M$ is compact, as we will assume throughout the paper, then the Reeb flow of $\eta$ yields a natural torus action on $M$: the closure $T$ of the Reeb flow in the isometry group $\Isom(M,g)$ of $(M,g)$ is a connected abelian Lie subgroup. As $\Isom(M,g)$ is a compact Lie group by the Myers-Steenrod theorem~\cite{Myers Steenrod}, so is the closed subgroup $T$, which implies that $T$ is a torus.

Because the Reeb flow preserves $\eta$ and $g$, the $T$-action preserves $\eta$ and $g$ by continuity. We say that the characteristic foliation (respectively the Sasakian structure) is \emph{regular} if $\dim T=1$ and the $T$-action is free, \emph{quasi-regular} if $\dim T=1$ and the $T$-action is only locally free, and \emph{irregular} if $\dim T>1$.

\section{Basic cohomology of transversely K\"ahler foliations} \label{sec:basiccoh}

\subsection{Transversely holomorphic foliations}\label{sec:thf}
We recall the definition of transversely holomorphic foliations.  A {\em transversely holomorphic Haefliger cocycle} of complex codimension $r$ on a manifold $M$ is a triple $(\{U_i\},\{\pi_i\},\{\gamma_{ij}\})$ consisting of
\begin{enumerate}
\item an open covering $\{U_i\}$ of $M$,
\item submersions $\pi_i:U_i\to \CC^r$,
\item biholomorphic transition functions $\gamma_{ij}:\pi_j(U_i\cap U_j)\to \pi_i(U_i\cap U_j)$ such that $\pi_i=\gamma_{ij}\circ \pi_j$.
\end{enumerate}
Two transversely holomorphic Haefliger cocycles on $M$ are said to be equivalent if their union is a transversely holomorphic Haefliger cocycle on $M$. A {\em transversely holomorphic foliation} of complex codimension $r$ is defined to be an equivalence class of transversely holomorphic Haefliger cocycles of complex codimension $r$.

\begin{rem}
Each transversely holomorphic foliation of codimension $r$ has an underlying real foliation of codimension $2r$, whose restriction to $U_{i}$ is defined by the fibers of $\pi_i:U_i\to \CC^r$. The above three conditions mean that $\F$ is transversely modeled on $\CC^r$ and that the transition functions are biholomorphic.
\end{rem}

A transversely holomorphic foliation $\F$ is called {\em transversely Hermitian} if there is a Hermitian metric $h_{i}$ on $\pi_{i}(U_{i})$ such that $\gamma_{ij}^{*}h_{i} = h_{j}$.
$\F$ is called {\em transversely K\"{a}hler} if the fundamental form $\omega_{i}={\rm Im}(h_i)$ of $h_{i}$ on $\pi_{i}(U_{i})$ is closed for every $i$.

The normal bundle $\nu\F=TM/T\F$ of a transversely holomorphic foliation has a natural complex structure: for each point $x\in U_{i}$, we get an isomorphism $\nu_x\F\to T_{\pi_{i}(x)}\CC^r$ induced from $\pi_{i}$. The complex structure induced on $\nu_x\F$ from $T_{\pi_{i}(x)}\CC^r$ is independent of $i$ because the transition functions $\gamma_{ij}$ are biholomorphic. So the normal bundle $\nu \F$ has a complex structure $J$, which is called the {\em transverse complex structure} of $\F$. In a similar way, the normal bundle of a transversely Hermitian foliation has a natural Hermitian metric called the {\em transverse Hermitian metric} of $\F$.

These notions are important in this article mainly because of the following well-known fact, see e.g.~{\cite[Section~2]{Boyer Galicki Nakamaye2}}:
\begin{lemma}\label{lem:ReebtrKaehler} The characteristic foliation of a Sasakian manifold $(M,\eta,g)$ is transversely K\"{a}hler, with fundamental form $d\eta$.
\end{lemma}

\subsection{Basic Dolbeault cohomology}\label{subsec:BasicDolbeault}

Differential forms are $\CC$-valued throughout this paper. Let $\F$ be a transversely holomorphic foliation of $M$ given by a cocycle $(\{U_i\},\{\pi_i\},\{\gamma_{ij}\})$. A $k$-form $\alpha$ on $M$ is called {\em basic} if for every $i$ there exists a $k$-form $\alpha_{i}$ on $\pi_{i}(U_{i})$ such that $\alpha|_{U_{i}} = \pi_{i}^{*}\alpha_{i}$. Since the differential commutes with $\pi_{i}^{*}$, the set $\Omega^\bullet(M,\F)$ of basic forms is a subcomplex of $(\Omega^\bullet(M),d)$. The cohomology of $(\Omega^\bullet(M,\F),d)$ is called the {\em basic cohomology} of $(M,\F)$ and denoted by $H^{\bullet}(M,\F)$. The dimension of $H^{k}(M,\F)$ is called the $k${\em -th basic Betti number} of $(M,\F)$ and denoted by $b^{k}(M,\F)$.

We will recall the definition of basic Dolbeault cohomology of $\F$. A $k$-form on $M$ is called a {\em basic} $(p,q)$-{\em form} if for every $i$ there exists a $(p,q)$-form $\alpha_{i}$ on $\pi_{i}(U_{i})$ such that $\alpha|_{U_{i}} = \pi_{i}^{*}\alpha_{i}$. We denote the set of basic $(p,q)$-forms on $(M,\F)$ by $\Omega^{p,q}(M,\F)$. We have a canonical decomposition $\Omega^k(M,\F)=\bigoplus_{p+q=k} \Omega^{p,q}(M,\F)$. On $U_{i}$, we can decompose $d\alpha$ as
\begin{equation}\label{eq:doldiff0}
d\alpha|_{U_{i}} = \pi_{i}^{*} d\alpha_{i} = \pi_{i}^{*} \partial\alpha_{i} + \pi_{i}^{*} \bdel \alpha_{i}\;,
\end{equation}
where $\partial\alpha_{i}$ is a $(p+1,q)$-form on $\pi_{i}(U_{i})$ and $\bdel \alpha_{i}$ is a $(p,q+1)$-form on $\pi_{i}(U_{i})$. Since the transition functions $\gamma_{ij}$ are biholomorphic, this gives well-defined basic forms $\partial_b \alpha$ and $\bdel_b \alpha$, so we obtain the {\em basic Dolbeault operator} and its complex conjugate
\begin{equation}\label{eq:doldiff1}
{\renewcommand\arraystretch{1.5}
\begin{array}{l}
\bdel_{b} \colon \Omega^{p,q}(M,\F) \longrightarrow \Omega^{p,q+1}(M,\F)\;, \\
\partial_{b} \colon \Omega^{p,q}(M,\F) \longrightarrow \Omega^{p+1,q}(M,\F)\;.
\end{array}}
\end{equation}

The differential complex $(\Omega^{p,\bullet}(M,\F),\delbar_b)$ is called the ($p$-{\em th}) {\em basic Dolbeault complex} of $(M,\F)$. Its cohomology is called the ($p$-{\em th}) {\em basic Dolbeault cohomology} of $(M,\F)$ which is denoted by $H^{p,\bullet}(M,\F)$. The dimension
$$
h^{p,q}(M,\F):=\dim H^{p,q}(M,\F)
$$
is called the $(p,q)${\em -th basic Hodge number} of $\F$.

For basic $2$-forms on $(M,\F)$, positivity (resp., negativity) is defined in a way analogous to the positivity (resp., negativity) of $2$-forms on complex manifolds. We recall

\begin{defn}\label{def:pos}
A compact Sasakian manifold $(M, \eta,g)$ is called \emph{positive} (resp., \emph{negative}) if the basic first Chern class of the normal bundle of the characteristic foliation (see~\cite[Section~3.5.2]{ElKacimi}) is represented by a positive (resp., negative) basic $(1,1)$-form. If the basic first Chern class is presented by the trivial form, then $(M, \eta,g)$ is said to be \emph{null}.
\end{defn}

\begin{ex}\label{ex:basicDolbeaultdistinguishes} As mentioned in the introduction, on a fixed compact manifold, the basic Betti numbers of the characteristic foliation of any Sasakian structure cannot distinguish Sasakian structures~\cite[Theorem~7.4.14]{Boyer Galicki}. Let us give an example that the basic Hodge numbers can distinguish different Sasakian structures. Consider $M=21\# (S^2\times S^3)$, the $21$-fold connected sum of $S^2\times S^3$. By~\cite[Example~10.3.10]{Boyer Galicki}, $M$ admits a regular Sasakian structure as an $S^1$-bundle over a $K3$ surface, so that $h^{2,0}(M,\F)=h^{0,2}(M,\F)=1$ and $h^{1,1}(M,\F)=20$. On the other hand, $M$ also admits positive Sasakian structures, for example given by the link of the weighted polynomial $z_0^{22}+z_1^{22} +z_2^{22}+z_0z_3$ with weights $(1,1,1,21)$, see \cite[Example~10.3.7]{Boyer Galicki}, in particular the first line of the table on p.~356. By \cite[Proposition~9.6.3]{Boyer Galicki} (see also \cite[Proposition~7.5.25]{Boyer Galicki}), this Sasakian structure is indeed positive. By~\cite[Theorem~1.2]{Nozawa}, we have $h^{2,0}(M,\F)=h^{0,2}(M,\F)=0$ for any positive Sasakian structure on $M$, and consequently, $h^{1,1}(M,\F) = b^{2}(M,\F) = 22$. Note that Gomez~\cite{Gomez} constructed negative Sasakian structures on $M$, but we do not know their basic Hodge numbers.
\end{ex}

\subsection{Transverse Hodge theory for transversely K\"{a}hler foliations}\label{sec:HodgetrKaehler}

We recall the transverse Hodge theory for transversely K\"{a}hler foliations due to El Kacimi-Alaoui~\cite{ElKacimi}. The transverse Hermitian metric $h$ on the normal bundle $\nu \F$ of $\F$ determines a Hodge star operator $\stb : \wedge^{\bullet} \nu_{x}^{*}\F_{\CC} \longrightarrow \wedge^{2n-\bullet} \nu_{x}^{*} \F_{\CC}$ on the complexification $\nu \F_{\CC}$ of $\nu \F$. Since the bidegree decomposition of $\wedge^{\bullet} \nu_{x}^{*}\F_{\CC}$ (described in more detail in Section~\ref{sec:extension}) is preserved  at each point $x$ on $M$, we get
$$
\stb : \Omega^{p,q}(M,\F) \longrightarrow \Omega^{n-q,n-p}(M,\F)\;.
$$
Define $\bdel_b^* = -\stb \partial_b \stb$ and consider the basic Laplacian
$$
\Delta_{b}:=\bdel_b \bdel_b^*+\bdel_{b}^* \bdel_{b} : \Omega^{p,q}(M,\F) \longrightarrow \Omega^{p,q}(M,\F)\;.
$$

For a Riemannian foliation of codimension $m$ on a connected closed manifold, $H^{m}(M,\F)=\CC$ or $H^{m}(M,\F)=0$ by a result of~\cite{ESH}. If $H^{m}(M,\F)=\CC$ then $(M,\F)$ is called {\it homologically orientable}. El~Kacimi-Alaoui~\cite{ElKacimi} has generalized the classical Hodge theory for K\"{a}hler manifolds to transversely K\"ahler foliations as follows:

\begin{thm}[{\cite[Th\'eor\`eme 3.3.3]{ElKacimi}}]\label{thm:EKA}
If $\F$ is a transversely Hermitian foliation on a compact manifold $M$, then we have a decomposition
$$
\Omega^{p,q}(M,\F) = \ker \Delta_{b} \oplus \Image \bdel_b \oplus \Image \bdel^{*}_b\;.
$$
In particular, there is an isomorphism $H^{p,q}(M,\F) \cong \ker \Delta_{b}$.
\end{thm}

For a transversely K\"ahler foliation, the K\"{a}hler identities for $\bdel_b$ and $\bdel^{*}_b$ on basic forms are reduced to the K\"ahler identities on each domain $\pi_{i}(U_{i})$ in $\CC^{r}$. So Theorem~\ref{thm:EKA} has the following consequence:
\begin{thm}[{\cite[Th\'eor\`eme 3.4.6]{ElKacimi}}]\label{thm:hodgedecomp}
If $\F$ is a
homologically orientable
transversely K\"ahler foliation on a compact manifold $M$, then we have a basic Hodge decomposition
$$
H^k(M,\F)=\bigoplus_{p+q=k}H^{p,q}(M,\F)\;.
$$
In particular, we have
$$
b^k(M,\F)=\sum_{p+q=k}h^{p,q}(M,\F)\;.
$$
\end{thm}

These theorems are relevant in our situation, because the characteristic foliation $\F$ of a Sasakian manifold $(M^{2n+1},\eta,g)$ is a transversely K\"ahler foliation by Lemma~\ref{lem:ReebtrKaehler}. If $M$ is compact and connected, $\F$ is moreover homologically orientable, since $[d\eta]^n$ is nontrivial in $H^n(M,\F)$, as one can easily see (e.g.~\cite[Proposition~7.2.3]{Boyer Galicki}). Alternatively, homological orientability can also be shown with a result~\cite[Th\'{e}or\`{e}me~A]{MolSer} of Molino-Sergiescu from the fact that the Reeb flow is isometric.

\section{Invariance of basic Hodge numbers under deformations of type I} \label{sec:invarianceofhpq}

\subsection{The cone of Reeb vector fields}\label{sec:cone}

Let $(M,\eta,g)$ be a compact, connected Sasakian manifold with Reeb vector field $\xi$. Let $(\cD,J)$ be the underlying CR structure of $(\eta,g)$, where $\cD=\ker\eta$ and $J$ is determined by $g(X,Y) = d\eta(X,JY)$. The group $\AUT(M,\eta,g)$ of diffeomorphisms that preserve $\eta$ and $g$ is a compact Lie group by Myers-Steenrod's theorem~\cite{Myers Steenrod}. The group $\CR(\cD,J)$ of diffeomorphisms that respect $\cD$ and $J$ is a Lie group by~\cite[l.~24 on p.~245]{ChernMoser} and we denote its Lie algebra by $\Lcr(\cD,J)$. If $M$ is not CR diffeomorphic to the sphere with standard CR structure, then $\CR(\cD,J)$ is compact by a theorem of Schoen~\cite{Schoen}. Clearly, we have $\AUT(M,\eta,g)\subset \CR(\cD,J)$.

Following~\cite{BGS}, we denote by ${\mathcal S}(\mathcal D,J)$ the set of Sasakian structures on $M$ with underlying CR structure $(\mathcal D,J)$. Define the convex cone
$$
\mathfrak{cr}^+(\mathcal D,J)=\{\zeta\in \mathfrak{cr}(\mathcal D,J)\mid \eta(\zeta)>0\}\;.
$$
By~\cite[Lemma~6.4]{BGS}, the map ${\mathcal S}(\mathcal D,J)\to \mathfrak{cr}^+(\mathcal D,J)$ sending a Sasakian structure to its Reeb vector field is a bijection.

We will use this cone in the next lemma.
\begin{lemma}\label{lem:cone}
Assume that the CR diffeomorphism group $\mathfrak{CR}(\mathcal D,J)$ of the underlying CR structure $(\mathcal D,J)$ of a compact Sasakian manifold $(M,\eta,g)$ is a compact Lie group. Then there exists a Sasakian structure $(\eta_{0},g_{0})$ in ${\mathcal S}(\mathcal D,J)$ such that for any Sasakian structure $(\eta_{1},g_{1})$ in ${\mathcal S}(\mathcal D,J)$, there exists a smooth family $\{(\eta_{s},g_{s})\}_{0 \leq s \leq 1}$ in ${\mathcal S}(\mathcal D,J)$ and  a torus $T$ in  $\mathfrak{CR}(\mathcal D,J)$ such that
\begin{enumerate}
\item $T\subset\AUT(\eta_{s},g_{s})$ for any $s$ and
\item the Reeb flow of $\eta_{s}$ is a one-parameter subgroup of the $T$-action.
\end{enumerate}
\end{lemma}

\begin{proof}
By~\cite[Proposition 4.4]{BGS} and its proof, averaging the contact form $\eta$ with the action of the identity component $G=\mathfrak{CR}_0(\mathcal D,J)$ of $\mathfrak{CR}(\mathcal D,J)$ gives a new $G$-invariant Sasakian structure $(\eta_{0},g_{0})$ in ${\mathcal S}(\mathcal D,J)$ such that the identity component of $\AUT(\eta_0,g_0)$ is $G$. Therefore, the Reeb vector field $\xi_{0}$ of $(\eta_{0},g_{0})$ belongs to $\mathfrak{cr}^+(\mathcal D,J)$ and is contained in the center of $\mathfrak{cr}(\mathcal D,J)$. Now for any Sasakian structure $(\eta_{1},g_{1})$ in ${\mathcal S}(\mathcal D,J)$ and the corresponding element $\xi_{1} \in\mathfrak{cr}^+(\mathcal D,J)$, let $T\subset G$ be a torus whose Lie algebra contains the span of $\xi_{0}$ and $\xi_{1}$. The line segment from $\xi_{0}$ to $\xi_{1}$ lies in $\mathfrak{cr}^+(\mathcal D,J)$ by convexity, and for the corresponding family of Sasakian structures $(\eta_s,g_s)$ on $M$ we have $T\subset \AUT(\eta_s,g_s)$. This concludes the proof.
\end{proof}

\subsection{The statement of the main results of Section~\ref{sec:invarianceofhpq}}\label{sec:st}

Sections~\ref{sec:reduction} through~\ref{sec:proofofuppersemic} will be devoted to proving the following rigidity theorem:

\begin{thm}\label{thm:Invariance3}
Let $\{(\eta_{s},g_{s})\}_{0 \leq s \leq 1}$ be a smooth family of Sasakian structures on a compact manifold $M$ with characteristic foliations $\F_{s}$. Assume that there exists a smooth family of actions $\{\phi_s:T\to\AUT(\eta_s,g_s)\}_{0 \leq s \leq 1}$ of a torus $T$ on $M$ which contains the Reeb flow of $\eta_{s}$ as a one-parameter subgroup for any $s$. Then the basic Hodge numbers of the Reeb flows of $(\eta_{s},g_{s})$ are independent of $s$.
\end{thm}
\begin{rem}
We could have also stated the theorem for an arbitrary compact, connected Lie group instead of a torus. But because the Reeb flow of each $\eta_s$ would then be contained in the center of the Lie group, this would not be more general.
\end{rem}

This theorem is trivial for Sasakian structures on real cohomology $(2n+1)$-spheres because of the following well-known proposition.
\begin{prop}\label{prop:sp}
For any Sasakian structure on a real cohomology $(2n+1)$-sphere with characteristic foliation $\F$, we get
\begin{equation}
h^{p,q}(M,\F) =
\begin{cases}
1 & \text{if}\,\,\,\, 0 \leq p=q \leq n\;, \\
0 & \text{otherwise}\;.
\end{cases}
\end{equation}
\end{prop}
\begin{proof}
By the Gysin sequence~\cite{Saralegui} (Equation~(7.2.1) on p.~215 of \cite{Boyer Galicki}) for the characteristic foliation of any Sasakian structure on $M$, we get $H^{\bullet,\bullet}(M,\F)=\CC[z]/(z^{n+1})$, where $z$ corresponds to $d\eta$. Since $d\eta$ is of bidegree $(1,1)$, the result follows.
\end{proof}

We obtain the following corollary of Theorem~\ref{thm:Invariance3}:

\begin{thm}\label{thm:invtypeI}
Two Sasakian structures on a compact manifold with the same underlying CR structures have the same basic Hodge numbers.
\end{thm}
\begin{proof}
If $M$ is diffeomorphic to $S^{2n+1}$, then the statement follows from Proposition~\ref{prop:sp}. Assume now that $M$ is not diffeomorphic to $S^{2n+1}$. Then, as we have already remarked, the CR diffeomorphism group of the underlying CR structure is a compact Lie group. So Lemma~\ref{lem:cone} and Theorem~\ref{thm:Invariance3} imply the result.
\end{proof}

As already noted in the introduction, Theorem~\ref{thm:invtypeI} is equivalent to the following:
\begin{thm}\label{thm:defI}
The basic Hodge numbers of a compact Sasakian manifold are invariant under deformations of type I.
\end{thm}

We have the following strengthening of Theorem~\ref{thm:Invariance3}.
\begin{thm}\label{thm:Invariance2}
Let $\{(\eta_{s},g_{s})\}_{0 \leq s \leq 1}$ be a smooth family of Sasakian structures on a compact manifold $M$ with CR structures $(\cD_s,J_s)$ and characteristic foliations $\F_{s}$. Assume that there exists a smooth family of actions $\{\phi_s:H\to\CR(\cD_s,J_s)\}_{0 \leq s \leq 1}$ of a compact connected Lie group $H$ on $M$ which for each $s$ contains the Reeb flow of $\eta_{s}$ as a one-parameter subgroup. Then the $h^{p,q}(M,\F_{s})$ are independent of $s$.
\end{thm}
\begin{proof}[Proof of Theorem~\ref{thm:Invariance2} by Theorem~\ref{thm:Invariance3}]
By averaging each $\eta_{s}$ under the respective $H$-action, we get a smooth family $\{(\eta_{s}',g'_{s})\}_{0 \leq s \leq 1}$ of $H$-invariant Sasakian structures on $M$ such that the underlying CR structure of $(\eta_{s}',g'_{s})$ is the same as $(\eta_{s},g_{s})$ for any $s$ (see~\cite[the proof of Proposition~4.4]{BGS}). Theorem~\ref{thm:invtypeI} therefore implies that the basic Hodge numbers of $(\eta_s,g_s)$ and $(\eta_s',g_s')$ are the same. Theorem~\ref{thm:Invariance3}, applied to the action of the center of $H$, implies that the basic Hodge numbers of the Reeb flows of $(\eta_s',g'_s)$ are independent of $s$.
\end{proof}

We will reduce the proof of Theorem~\ref{thm:Invariance3} in Section~\ref{sec:reduction} to showing the upper semi-continuity of the basic Hodge numbers. In Section~\ref{sec:extension} we will prepare for the proof of the upper-semicontinuity by extending the basic Laplacian of the characteristic foliation (which is naturally defined on basic forms) to a self-adjoint elliptic operator on some Hermitian vector bundle in such a way that the basic Hodge numbers are still encoded in its kernel, see Equation~\eqref{eq:harmonicinvhodge}. This allows us to apply a slight modification of the classical deformation theory of Kodaira-Spencer in Section~\ref{sec:proofofuppersemic} to obtain the desired upper-semicontinuity.

The following is a consequence of the proof of Theorem~\ref{thm:Invariance3}, and will be shown in Section~\ref{sec:DeformHodgedecomp}.

\begin{thm}\label{thm:Hodge}
Within the space of Sasakian structures with a fixed CR structure on a compact manifold $M$, the Hodge decomposition $H^k(M,\F)=\bigoplus_{p+q=k}H^{p,q}(M,\F)$ smoothly depends on the corresponding Reeb vector field.
\end{thm}

\subsection{Reduction to the upper semi-continuity}\label{sec:reduction}
In this subsection we show that to prove Theorem~\ref{thm:Invariance3} it is sufficient to show that $h^{p,q}(M,\F_{s})$ is upper semi-continuous with respect to $s$, i.e.,
$$
h^{p,q}(M,\F_{s}) \leq h^{p,q}(M,\F_{s_0})
$$
for $s$ sufficiently close to a fixed $s_0$. By the Hodge decomposition for homologically orientable transversely K\"{a}hler foliation on compact manifolds by El Kacimi-Alaoui, Theorem~\ref{thm:hodgedecomp} in this paper, we get
$$
\sum_{p+q = k} h^{p,q}(M,\F_{s}) = b^{k}(M,\F_{s})
$$
for every $s$ where $b^{k}(M,\F_{s})$ is the $k$-th basic Betti number of $\F_{s}$. Thus we get
\begin{equation}\label{eq:h}
b^{k}(M,\F_{s}) = \sum_{p+q = k} h^{p,q}(M,\F_{s}) \leq \sum_{p+q = k} h^{p,q}(M,\F_{s_0}) = b^{k}(M,\F_{s})\;.
\end{equation}
Since $b^{k}(M,\F_{s})$ is constant with respect to $s$ by~\cite[Theorem 7.4.14]{Boyer Galicki}, we get an equality in~\eqref{eq:h}. Thus we have $h^{p,q}(M,\F_{s}) = h^{p,q}(M,\F_{s_0})$ for $s$ sufficiently close to $s_0$. Therefore $h^{p,q}(M,\F_{s})$ is locally constant with respect to $s$ and hence constant.

\subsection{Extension of basic Laplacians}\label{sec:extension}

Let $(M^{2n+1},\eta,g)$ be a compact Sasakian manifold with Reeb field $\xi$ and characteristic foliation $\F$. We can now decompose the complexification $\nu\F_{\CC}$ of the normal bundle $\nu\F$ of $\F$ into the $\sqrt{-1}$- and $-\sqrt{-1}$-eigenspace bundles of the transverse complex structure $J$ as $\nu\F_{\CC} = \nu\F^{1,0}\oplus \nu\F^{0,1}$. Similarly, we get a decomposition of the complexified conormal bundle $\nu^*\F_{\CC} =\nu^*\F^{1,0}\oplus \nu^*\F^{0,1}$, and we define
$$
E^{p,q}:=\bigwedge^p\nu^*\F^{1,0}\oplus \bigwedge^q\nu^*\F^{0,1}\;.
$$
Note that we have a decomposition
\begin{equation}\label{eq:decomptangvert}
\Omega^k(M) = \bigoplus_{u+p+q=k} \bigwedge^u T^*\F_\CC \otimes E^{p,q}\;,
\end{equation}
which depends on the metric. Let
$$
\Omega^{p,q}:=C^{\infty}(E^{p,q})
$$
be the space of sections of $E^{p,q}$. It is easy to see that $\Omega^{p,q}(M,\F)$ is a subspace of $\Omega^{p,q}$ determined as follows:
$$
\Omega^{p,q}(M,\F) = \{ \alpha \in \Omega^{p,q} \, | \, L_{X}\alpha = 0, \forall X \in C^{\infty}(T\F) \} =  {\Omega^{p,q}}^{T}\;,
$$
where $T$ is the closure of the Reeb flow mentioned in Section~\ref{sec:torus}. The basic Dolbeault operator and its conjugate in Equations~\eqref{eq:doldiff1} are extended to
\begin{align*}
\bdel \colon \Omega^{p,q} \to \Omega^{p,q+1}\quad  \text{ and }  \quad \partial \colon \Omega^{p,q} \to \Omega^{p+1,q},
\end{align*}
where $\bdel$ (resp., $\partial$) is the composition of $d|\Omega^{p,q}$ with the orthogonal projection to $\Omega^{p,q+1}$ (resp., $\Omega^{p+1,q}$) with respect to the decomposition~\eqref{eq:decomptangvert}.

Let $\stb:\Omega^{p,q}\to\Omega^{n-q,n-p}$ be the transverse Hodge star operator with respect to $h$. We obtain a Hermitian metric on $\Omega^{p,q}$ by
$$
\langle \alpha,\beta\rangle=\int_M\eta\wedge\alpha\wedge \stb\overline\beta\;.
$$
Note that the contact form $\eta$ is not a transverse form and the wedge-product $\eta\wedge\alpha\wedge \stb\overline\beta$ is taken in $\Omega^\bullet(M)$.
\begin{lemma}
The adjoint operator $\bdel^*$ of $\bdel:\Omega^{p,q-1}\to\Omega^{p,q}$ with respect to $\langle \cdot, \cdot \rangle$ is given by $\bdel^*=-\stb\partial\stb$. In other words,
\begin{equation}\label{eq:s}
\langle \bdel\alpha_1, \alpha_2 \rangle=
 \langle \alpha_1,-\stb\partial \stb\alpha_2\rangle
\end{equation}
for $\alpha_{1}\in \Omega^{p,q-1}$ and $\alpha_{2}\in \Omega^{p,q}$.
\end{lemma}

\begin{proof}
\begin{align*}
   & d(\eta\wedge\alpha_1 \wedge \stb\overline\alpha_2) \\
=\, & d\eta \wedge \alpha_1 \wedge \stb\overline\alpha_2  - \eta\wedge d\alpha_1 \wedge \stb\overline\alpha_2  + (-1)^{p+q} \eta\wedge \alpha_1 \wedge d\stb\overline\alpha_2  \\
=\, & -\eta\wedge \bdel \alpha_1 \wedge \stb\overline\alpha_2  + (-1)^{p+q}\eta\wedge \alpha_1 \wedge \bdel\stb\overline\alpha_2\\
=\, & -\eta\wedge \bdel \alpha_1 \wedge \stb\overline\alpha_2  - \eta\wedge \alpha_1 \wedge \stb(\stb\bdel\stb)\overline\alpha_2\;.
\end{align*}
The first summand of the second line is zero, since $d\eta$ is of bidegree $(1,1)$ and $\alpha_1 \wedge \stb\overline\alpha_2$ of bidegree $(n,n-1)$.  The second equality follows from the assumption on the bidegrees of $\alpha_1$ and $\alpha_2$ and from the fact that the product of $\eta$ and the tangential part of $d\alpha_i$ is zero. Now, by applying $\int_{M}$ to the first and last line of the equation, the theorem of Stokes implies~\eqref{eq:s}.
\end{proof}
Consequently the basic Laplacian $\Delta_{b} : \Omega^{p,q}(M,\F) \to \Omega^{p,q}(M,\F)$ is naturally extended to $\Delta_b : \Omega^{p,q} \to \Omega^{p,q}$ via $\Delta_b=\bdel\,\bdel^* + \bdel^*\bdel$, and the extension is self-adjoint with respect to $\langle \cdot, \cdot \rangle$. By Theorem~\ref{thm:EKA},
$$
h^{p,q}(M,\F)=\dim(\ker\Delta_b|\Omega^{p,q}(M,\F))\;.
$$
Note that $\Delta_b$ is only transversely strongly elliptic. For our later application of a theorem of Kodaira-Spencer we need a strongly elliptic operator acting on the space of sections of a Hermitian bundle. We define it by
$$
D:=\mathcal L_{\xi}\mathcal L_{\xi} - \Delta_b
$$
on $\Omega^{p,q}=C^\infty(E^{p,q})$. Note that $\mathcal L_{\xi}$ respects $\Omega^{p,q}$.
\begin{rem} This definition is motivated by a construction El Kacimi-Alaoui and Hector used to prove the transverse Hodge decomposition for Riemannian foliations~\cite[p.~224]{ElKacimiHector}; see also~\cite[p.~82]{ElKacimi}.
\end{rem}
\begin{lemma}
The differential operator $D$ is strongly elliptic and self-adjoint.
\end{lemma}
\begin{proof}
We can take a chart $(t, x_{1}, y_{1}, \ldots, x_{q}, y_{q})$ around any point $z$ of $M$ such that $\xi = \partial /\partial t$ and such that $(x_{1}, y_{1}, \ldots, x_{q}, y_{q})$ is a transverse holomorphic chart and $\{\partial/\partial x_{i}, \partial/ \partial y_{i}\}$ is an orthonormal basis of $(T_{z}\F)^{\perp}$. Then $D = \partial^{2} /\partial^{2} t - \Delta_{b}$ on this chart. For $\zeta\in T^{*}_{z}M$, let $\sigma_\zeta(P)\in \End(E^{p,q}_{z})$ be the symbol of $P$ at $\zeta$, where $P = \partial^{2} /\partial^{2} t$ or $\Delta_{b}$. For $\zeta = \zeta_{0}dt + \sum_{i=1}^{n}\zeta_{2i-1}dx_{i} + \sum_{i=1}^{n}\zeta_{2i}dy_{i}$, we have $\sigma_\zeta(\partial^{2} /\partial^{2} t)=\zeta_0^2\cdot\id_{E^{p,q}_{z}}$, and since $\Delta_{b}$ is the transverse Laplacian and independent of $t$, $\sigma_\zeta(\Delta_b)= - \frac{1}{2}\big(\sum_{i=1}^{2q} \zeta_{i}^{2}\big) \id_{E^{p,q}}$ (see, for example, \cite[Lemma~5.19]{Voisin}). We obtain $\sigma_\zeta(D)=\big(\zeta_{0}^{2} + \frac{1}{2}\sum_{i=1}^{2q} \zeta_{i}^{2}\big) \id_{E^{p,q}}$ as the sum of the symbols of the two operators, thereby implying strong ellipticity of $D$.

We know already that $\Delta_b$ is self-adjoint, so it remains to show that $\mathcal L_\xi$ is skew-symmetric. For $\alpha_1,\alpha_2\in\Omega^{p,q}$ we have
$$
\mathcal L_\xi(\eta\wedge\alpha_1\wedge\stb\overline\alpha_2)=\eta\wedge\mathcal L_\xi\alpha_1\wedge\stb\overline\alpha_2+ \eta\wedge\alpha_1\wedge\stb\mathcal L_\xi\overline\alpha_2\;,
$$
because $\mathcal L_\xi\eta=0$ and $\mathcal L_\xi\stb=\stb\mathcal L_\xi$ (as $\xi$ is Killing). It therefore suffices to show that the left hand side is zero. Now
$$
\mathcal L_\xi(\eta\wedge\alpha_1\wedge\stb\overline\alpha_2)=d\iota_\xi(\eta\wedge\alpha_1\wedge\stb\overline\alpha_2)
=d(\alpha_1\wedge\stb\overline\alpha_2)\;,
$$
because the $\alpha_i$ are horizontal and $\eta(\xi)\equiv 1$. The form on the right is a basic $(2n+1)$-form and therefore zero.
\end{proof}

On basic forms, i.e., sections of $E^{p,q}$ invariant under the Reeb flow,  $D$ coincides with the basic Laplacian $\Delta_{b}$. By the basic Hodge decomposition (Theorem~\ref{thm:EKA}), this means the basic Hodge numbers $h^{p,q}(M,\F)$ are given by
\begin{equation}\label{eq:harmonicinvhodge}
h^{p,q}(M,\F)= \dim (\ker \Delta_{b}) = \dim(\ker \big( D|{\Omega^{p,q}}^H \big) ) =\dim(\ker D)^H\;,
\end{equation}
where $H$ is the Reeb flow. In fact, this equality is also true for any connected Lie subgroup $H$ of the automorphism group $\AUT(\eta,g)$ of the Sasakian structure whose Lie algebra contains the Reeb field. This is because any Sasaki automorphism commutes with the Reeb flow, is transversely isometric and holomorphic, and therefore commutes with $\Delta_b$, thus mapping a harmonic basic form to a harmonic basic form. Since $H$ as a connected Lie group acts trivially on basic Dolbeault cohomology, it fixes harmonic basic forms.

\subsection{Proof of the upper semi-continuity}\label{sec:proofofuppersemic}

We consider a smooth family $\F_{s}$ of Reeb flows of Sasakian structures given as in Theorem~\ref{thm:Invariance3}. We have a family of actions $\{H\to \AUT(\eta_s,g_s)\}$ of a connected compact Lie group $H$ on $M$ that contains the Reeb flow of $\eta_{s}$ as one-parameter subgroups. We consider the Hermitian bundle $E^{p,q}_{s} $ over $M$ with the operator $D_s$ on $C^{\infty}(E^{p,q}_{s})$. This is a smooth family of strongly elliptic, self-adjoint differential operators.
We have already mentioned
$$
h^{p,q}(M,\F_{s}) = \dim (\ker D_{s})^H\;.
$$

By the spectral theorem for formally self-adjoint strongly elliptic differential operators (see~\cite[Theorem 1]{KS}), there exists a complete orthonormal system of eigensections $\{e_{sh}\}_{h=1}^{\infty}$ of $D_{s}$, such that their eigenvalues $\lambda_h(s)$ constitute for each $s$ an ascending sequence in $[0,\infty)$ whose unique accumulation point is $\infty$.

In the following, we fix a parameter $s_0$. Take an integer $k_{0}$ so that
$$
\{ h \,\, | \,\, \lambda_{h}(s_0) = 0\} = \{h \,\, | \,\, 1 \leq h \leq k_{0}\}\;.
$$

Let $\E_{s}$ be the subspace of $C^{\infty}(E_{s}^{p,q})$ spanned by $\{e_{sh}\}_{h=1}^{k_{0}}$. It is therefore a sum of eigenspaces of $D_s$.

\begin{lemma}\label{lem:E_s}
There exists an open neighborhood $U'$ of $s_0$ in the parameter space $[0,1]$ and smooth families $\{f_{sh}\}_{h=1}^{k_{0}}$ of sections of $E_{s}^{p,q}$ such that for each $s\in U'$ the subspace of $C^{\infty}(E_{s}^{p,q})$ spanned by $\{f_{sh}\}_{h=1}^{k_{0}}$ is equal to $\E_{s}$.
\end{lemma}

\begin{proof}
For $1 \leq h \leq k_{0}$, we extend $e_{s_0h}$ to a smooth family of sections $e'_{sh}$ of $E_{s}^{p,q}$ such that $e'_{s_0h} = e_{s_0h}$ by partition of unity on $M$. Choose a circle $C$ in $\CC$ centered at $0$ so that the closed disk which bounds $C$ contains only $0$ among the eigenvalues $\{\lambda_{h}(s_0)\}_{h=1}^{\infty}$ of $D_{s_0}$. By~\cite[Theorem 2]{KS}, $\lambda_{h}(s)$ is continuous with respect to $s$. Then there exists an open neighborhood $U$ of $s_0$ in the parameter space such that none of $\lambda_{h}(s)$ for any $h$ and any $s\in U$ lie on $C$. This assumption on $C$ shows that $\E_s$ is the sum of eigenspaces of $D_s$ corresponding to the eigenvalues $\lambda_1(s),\ldots,\lambda_{k_0}(s)$. Moreover, this condition implies that we can apply ~\cite[Theorem 3]{KS}:  by setting
$$
F_{s}(C)(f) = \sum_{1 \leq h \leq k_{0}} \langle f,e_{sh}\rangle_{s}e_{sh}
$$
for $f\in C^{\infty}(E_{s}^{p,q})$, we obtain a smooth family $F_{s}(C) : C^{\infty}(E_{s}^{p,q}) \to \E_s$ of projections, that is, $F_{s}(C)$ maps a smooth family of sections of $E_{s}^{p,q}$ to a smooth family of $\E_s$. Thus letting $f_{sh}=F_{s}(C)(e'_{sh})$ for $1 \leq h \leq k_{0}$, we get a smooth family $f_{sh}\in\E_s$. Since $\{f_{s_0h}\}_{h=1}^{k_{0}} = \{e_{s_0h}\}_{h=1}^{k_{0}}$ is linearly independent, there exists an open neighborhood $U'$ of $s_0\ in U$ such that $\{f_{sh}\}_{h=1}^{k_0}$ is linearly independent for each $s\in U'$.
\end{proof}

\begin{lemma}
The space $\E_{s}$ is an $H$-invariant subspace of $C^\infty(E_s^{p,q})$, and the dimension of the subspace $\E_{s}^H$ of $H$-fixed elements is constant with respect to $s$.
\end{lemma}

\begin{proof}
Since the $H$-action on $C^{\infty}(E_{s}^{p,q})$ commutes with $D_{s}$, the eigenspaces of  $D_{s}$ are invariant. Thus $\E_{s}$, a sum of eigenspaces of $D_{s}$ by definition, is invariant. By Lemma~\ref{lem:E_s}, $\bigcup_{s \in U'} \E_{s}$ forms a trivial smooth vector bundle $\E \to U'$ of rank $k_{0}$ over $U'$ with a smooth trivialization $\{f_{sh}\}$. The existence of a smooth trivialization of $\E$ implies that the $H$-actions on the $\E_{s}$ form a smooth fiberwise $H$-action on $\E$. It is well known that representations of a compact Lie group do not change the isomorphism class under smooth deformations (see, for example,~\cite[Proposition B.57]{GGK}), and therefore $\dim \E_{s}^H$ is constant with respect to $s$.
\end{proof}

We now come to the main goal of this subsection.
\begin{lemma}\label{lem:usc}
$\dim (\ker D_{s})^H$ is upper semicontinuous with respect to $s$ at $s_0$.
\end{lemma}
\begin{proof}
It is sufficient to show that $\dim (\ker D_{s_0})^H \geq \dim (\ker D_{s})^H$ for $s\in U'$. At $s=s_0$, we get $(\ker D_{s_0})^H = \E_{s_0}^H$. On the other hand, for any $s\in U'$ it is true that $\ker D_{s} \subset \E_{s}$ by construction of $\E_{s}$. Thus we get
\begin{equation}\label{eq:ats}
 (\ker D_{s})^H \subset \E_{s}^H\;.
\end{equation}
Since $\dim \E_{s_0}^H = \dim \E_{s}^H$ by the last lemma, we conclude the proof.
\end{proof}
This concludes the proof of Theorem~\ref{thm:Invariance3}.\smallskip

\subsection{Proof of Theorem~\ref{thm:Hodge}}\label{sec:DeformHodgedecomp}
$\E_s^H$ is the space of harmonic forms of bidegree $(p,q)$ and is naturally embedded in $\Omega^\bullet(M)$. To prove Theorem~\ref{thm:Hodge} it is sufficient to show that this embedding depends smoothly on $s$. Now Theorem~\ref{thm:Invariance3} implies that equality holds in~\eqref{eq:ats}. So $\E^H=\bigcup_{s\in [0,1]} \E_s^H$ is a smooth vector bundle over $[0,1]$ and hence trivial. The span of a trivializing basis field gives a smooth map of $\E^H\to \Omega^\bullet(M)$. This proves Theorem~\ref{thm:Hodge}.

\section{Transverse actions on foliated manifolds}\label{sec:transverseactionsonchar}

\subsection{Definition of transverse actions on foliated manifolds}\label{sec:transverseaction}

In this section, let us recall transverse actions on foliated manifolds introduced in~\cite{Alvarez Lopez} and related notions.

Let $\F$ be a foliation of a manifold $M$. By $\Xi(\F)$ we denote the space of differentiable vector fields on $M$ that are tangent to the leaves of $\F$. A vector field $X$ on $M$ is said to be {\em foliate} if for every $Y\in \Xi(\F)$ the Lie bracket $[X,Y]$ also belongs to $\Xi(\F)$. A vector field is foliate if and only if its flow maps leaves of $\F$ to leaves of $\F$, see~\cite[Proposition~2.2]{Molino}. The set $L(M,\F)$ of foliate fields is the normalizer of $\Xi(\F)$ in the Lie algebra $\Xi(M)$ of vector fields on $M$ and therefore a Lie sub-algebra of $\Xi(M)$. We call the projection of a foliate field $X$ to $TM/T\F$ a {\em transverse} field. The set $l(M,\F)=L(M,\F)/\Xi(\F)$ of transverse fields is also a Lie algebra inheriting the Lie bracket from $L(M,\F)$.
\begin{defn}[{\cite[Section 2]{Alvarez Lopez}}]\label{def:transverseaction}
A  {\em transverse action} of a finite-dimen\-sional Lie algebra $\mfg$ on the foliated manifold $(M,\F)$ is a Lie algebra homomorphism $\mfg\to l(M,\F)$.
\end{defn}
\noindent Given a transverse action of $\mfg$, we will denote the transverse field associated to $X\in \mfg$ by $\trX\in l(M,\F)$. If $\F$ is the trivial foliation by points, this notion coincides with the usual notion of an infinitesimal action on the manifold $M$.

If $\F$ is a Riemannian foliation, then a transverse field is called a \emph{transverse Killing field} if one (and hence all) of its representatives in $L(M,\F)$ leaves invariant the transverse metric. The set of transverse Killing fields form a Lie subalgebra ${\mathfrak{iso}}(M,\F,g)$ of $l(M,\F)$.
\begin{defn} If $\F$ is a Riemannian foliation, then a transverse action $\mfg\to {\mathfrak{iso}}(M,\F,g)$ is called \emph{isometric}.
\end{defn}

If $\F$ is a transversely holomorphic foliation, then a transverse field $X\in l(M,\F)$ is said to \emph{preserve the transverse complex structure of $\F$} if $L_{\trX}J=0$, where $J$ is the transverse complex structure of $\F$.

\begin{defn}
If $\F$ is a transversely holomorphic foliation, then a transverse action $\mfg\to l(M,\F)$ is called {\em holomorphic} if every $X$ preserves the transverse complex structure of $\F$.
\end{defn}

\subsection{The canonical transverse action on a Sasakian manifold}\label{sec:tractiononsas}

Let $(M,\eta,g)$ be a compact Sasakian manifold with characteristic foliation $\F$. The closure of the Reeb flow is denoted by $T$. We have already remarked in Section~\ref{sec:HodgetrKaehler} that the characteristic foliation $\F$ is transversely K\"ahler. By the commutativity of $T$, the $T$-action preserves $\F$. Hence there is a canonical transverse action on $(M,\F)$
\begin{equation}\label{eq:traction}
\mft/\RR \xi \longrightarrow l(M,\F)\;.
\end{equation}
By~\cite[Equation~(8.1.4)]{Boyer Galicki}, the natural transverse action of $\mfa=\mft/\RR$ is isometric and holomorphic.

As recalled in the next section, a more general class of Riemannian foliations called Killing foliations admit natural actions of Abelian Lie algebras in a similar way.

\subsection{The canonical transverse action on Killing foliations}\label{sec:traction}\label{sec:tractiononKilling}

The Molino sheaf $\mathcal{C}$ of a Riemannian foliation $(M,\F)$ is a locally constant sheaf of Lie algebras, whose stalks consist of certain local transverse fields. Precisely, a stalk of $\mathcal{C}$ consists of local transverse vector fields on $(M,\F)$ whose natural lifts to the orthonormal frame bundle $M^{1}$ of $(M,\F)$ commute with any global transverse field of $(M^{1},\F^{1})$, where $\F^{1}$ is the canonical horizontal lift of $\F$ (see \cite[Section~4]{Molino}).
\begin{defn} A Riemannian foliation is called a \emph{Killing foliation} if its Molino sheaf is globally constant.
\end{defn}
In particular, Riemannian foliations on simply connected manifolds are Killing.

Any global section of the Molino sheaf $\mathcal{C}$ of a Riemannian foliation $(M,\F)$ is a transverse field on $(M,\F)$ which commutes with any global transverse field of $(M,\F)$. So the space $\mfa$ of global sections of $\mathcal C$ is central in $l(M,\F)$, hence it is an abelian Lie algebra acting transversely on $(M,\F)$. Following~\cite{GT2010}, for a Killing foliation $\F$ the Lie algebra $\mfa$ is called the {\em structural Killing algebra} of $\F$. The orbits of the leaves under the action of the structural Killing algebra are the leaf closures~\cite[Theorem~5.2]{Molino}, see also~\cite[Section~4.1]{GT2010}.

\begin{ex}
By~\cite[Th\'{e}or\`{e}me~A]{MolSer}, the characteristic foliation of a Sasakian manifold is a Killing foliation. Its Killing algebra is identified with $\mfa=\mft/\RR$ in the last section by~\cite[Example~4.3]{GT2010}.
\end{ex}

We have seen that the canonical transverse action on a Sasakian manifold is isometric and holomorphic. This is more generally true for Hermitian Killing foliations:
\begin{lemma}\label{lem:transversisometric} If $\F$ is a transversely Hermitian Killing foliation, then the transverse $\mfa$-action on $(M,\F)$ is isometric and holomorphic.
\end{lemma}
\begin{proof} Let $\{U_i\}$ be a covering of $M$ by foliation chart domains with transversals $T_i\subset U_i$. The holonomy pseudogroup of $\F$ acts isometrically and holomorphically on the union of the $T_i$. By~\cite[p.~287]{Molino} any $X\in \mfa$ is an infinitesimal transformation of elements in the closure of the holonomy pseudogroup and hence acts isometrically and holomorphically.
\end{proof}

\section{Equivariant cohomology} \label{sec:eqcohom}

\subsection{$\mfg$-differential graded algebras and the Cartan model}

In this section, we recall the Cartan model of equivariant cohomology in the language of differential graded algebras. Let $\mfg$ be a finite-dimensional Lie algebra.

\begin{defn}\label{defn:g*}
A $\mfg$-{\em differential graded algebra} ($\mfg$-{\em dga}) is a $\ZZ$-graded algebra $A=\bigoplus A_k$ endowed with the following data: a derivation $d:A\to A$ of degree $1$ and  derivations $\iota_X:A\to A$ of degree $-1$ and $L_X:A\to A$ of degree $0$ for all $X\in \mfg$ (where $\iota_X$ and $L_X$ depend linearly on $X$) such that:
\begin{align*}
d^2&=0\;, &&& \iota_X^2&=0\;, &&& L_X&=d\iota_X+\iota_Xd\;, \\
[d,L_X]&=0\;, &&& [L_X,\iota_Y]&=\iota_{[X,Y]}\;, &&& [L_X,L_Y]&=L_{[X,Y]}\;.
\end{align*}
\end{defn}

\begin{ex}\label{ex:gstaraction}
An infinitesimal action of a finite-dimensional Lie algebra $\mfg$ on a manifold $M$, i.e., a Lie algebra homomorphism $\mfg\to \Xi(M);\, X\mapsto \trX$, induces a $\mfg$-dga structure on the de Rham complex $\Omega(M)$ with operators $\iota_X:=\iota_{\trX}$ and $L_X:=L_{\trX}$.
\end{ex}

The {\em Cartan complex} of $A$ is defined as
\begin{equation*}
C_\mfg(A):=(S(\mfg^*)\otimes A)^\mfg\;.
\end{equation*}
\noindent Here the superscript denotes the subspace of $\mfg$-invariant elements, i.e., those $\omega\in S(\mfg^*)\otimes A$ for which $L_X\omega=0$ for all $X\in \mfg$. The grading on the Cartan complex is given by the natural product grading on $S(\mfg^*)\otimes A$, where we assign the nonzero elements in $\mfg^*\subset S(\mfg^*)$ the degree $2$.

The differential $d_\mfg$ of the Cartan complex $C_\mfg(A)$ is defined by
$$
(d_\mfg \omega)(X)=d(\omega(X))-\iota_X(\omega(X))\;,
$$
where we consider an element in $C_\mfg(A)$ as a $\mfg$-equivariant polynomial map $\mfg\to A$. Now the equivariant cohomology of the $\mfg$-dga $A$ is defined as
$$
H_\mfg^\bullet(A):=H^\bullet(C_\mfg(A),d_\mfg)\;.
$$
There is a natural $S(\mfg^*)^\mfg$-algebra structure on $H^\bullet_\mfg(A)$. The $\mfg$-dga $A$ is called {\em equivariantly formal} if $H_\mfg^{\bullet}(A) \cong S^{\bullet}(\mfg^*)^\mfg\otimes H^{\bullet}(A)$ as a graded $S^{\bullet}(\mfg^*)^\mfg$-module.

\subsection{Equivariant basic cohomology}

Let $\F$ be a foliation on a manifold $M$ and $\Omega(M,\F)$ the basic de Rham complex of $(M,\F)$. For $X\in l(M,\F)$, $\iota_X$ and $L_X$ are well-defined derivations on $\Omega(M,\F)$.
\begin{prop}[{\cite[Proposition~3.12]{GT2010}}]\label{prop:structureofgdga}
A transverse action of a finite-dimensional Lie algebra $\mfg$ on a foliated manifold $(M,\F)$ induces the structure of a $\mfg$-dga on $\Omega(M,\F)$.
\end{prop}
\noindent The Cartan complex $C_\mfg(\Omega(M,\F))$ will also be denoted by $\Omega_\mfg(M,\F)$.
\begin{defn}[{\cite[Section 3.6]{GT2010}}]\label{def:eqformaltraction}
The {\em equivariant basic cohomology} of a transverse $\mfg$-action on $(M,\F)$ is defined as
$$
H^\bullet_\mfg(M,\F):=H^\bullet(\Omega_\mfg(M,\F))\;.
$$
The $\mfg$-action is called {\em equivariantly formal} if $\Omega(M,\F)$ is an equivariantly formal $\mfg$-dga, i.e., if $H^\bullet_\mfg(M,\F)\cong S^{\bullet}(\mfg^*)^\mfg\otimes H^\bullet(M,\F)$ as a graded $S^{\bullet}(\mfg^*)^\mfg$-module.
\end{defn}

\begin{rem}
Note that for the case of the foliation of $M$ by points, a transverse action of a finite-dimensional Lie algebra $\mfg$ is nothing but an ordinary infinitesimal action on $M$. If such an action is induced by an action of a compact connected Lie group $G$, the equivariant basic cohomology of the $\mfg$-action is the same as the usual Cartan model of the equivariant cohomology of the $G$-action on $M$.
\end{rem}

\subsection{Equivariant Dolbeault cohomology of $\mfg$-differential graded algebras}

Let $\mfg$ be a finite-dimensional Lie algebra. We define equivariant Dolbeault cohomology for certain bigraded $\mfg$-dgas. To formalize our argument, we introduce the following object:
\begin{defn}\label{defn:holdga}
A $\mfg$-{\em differential graded algebra of Dolbeault type} is a bigraded algebra $A=\bigoplus_{p,q\in \ZZ} A^{p,q}$ whose total algebra (defined by $A_k=\bigoplus_{p+q=k} A^{p,q}$) is a $\mfg$-dga such that
\begin{enumerate}
\item $d(A^{p,q}) \subset A^{p+1,q} \oplus A^{p,q+1}$,
\item $\iota_{X}(A^{p,q}) \subset A^{p-1,q} \oplus A^{p,q-1}$ for any $X\in \mfg$ and
\item $L_{X}(A^{p,q}) \subset A^{p,q}$ for any $X\in \mfg$.
\end{enumerate}
\end{defn}
\begin{rem}
$d(A_{k}) \subset A_{k+1}$, $\iota_{X}(A_{k}) \subset A_{k-1}$ and $L_{X}(A_{k}) \subset A_{k}$ are satisfied for any $X\in \mfg$ by definition of $\mfg$-dgas.
\end{rem}

Given a $\mfg$-dga $A$ of Dolbeault type, we may consider the following natural bigrading on the Cartan complex $C_\mfg(A)$:
$$
C_\mfg^{p,q}(A):=\bigoplus_{l+r=p,\, l+s=q} (S^l(\mfg^*)\otimes A^{r,s})^\mfg\;.
$$
In other words: the elements in $\mfg^*\subset S^{\bullet}(\mfg^*)$ are assigned the bidegree $(1,1)$. We split $d$ and $\iota_{X}$ for $X\in \mfg$ as $d=d^{1,0} + d^{0,1}$ and $\iota_{X}=\iota_X^{-1,0}+\iota_X^{0,-1}$, where the suffix denotes the bidegree. Then the equivariant differential $d_\mfg$ splits into its components of bidegree $(1,0)$ and $(0,1)$ as $d_\mfg=d_\mfg^{1,0} + d_\mfg^{0,1}$, where
\begin{align*}
(d_\mfg^{1,0} \omega)(X)&=d^{1,0}(\omega(X)) - \iota_X^{0,-1}(\omega(X))\;,\\
(d_\mfg^{0,1} \omega)(X)&=d^{0,1}(\omega(X)) - \iota_X^{-1,0}(\omega(X))\;.
\end{align*}
\begin{defn}
The {\em equivariant Dolbeault cohomology} of $A$ is defined as
$$
H_\mfg^{p,\bullet}(A):=H^{\bullet}(C^{p,\bullet}_\mfg(A),d_\mfg^{0,1})\;.
$$
$A$ is said to be \emph{Dolbeault equivariantly formal} if
\[
H^{\bullet,\bullet}_\mfg(A) \cong S^{\bullet}(\mfg^*)^\mfg\otimes H^{\bullet,\bullet}(A,d^{0,1})
\]
as a bigraded $S^{\bullet}(\mfg^*)^\mfg$-module.
\end{defn}

\subsection{Equivariant basic Dolbeault cohomology}
Let $\F$ be a transversely holomorphic foliation on a manifold $M$ of complex codimension $r$. Consider a transverse holomorphic action of a Lie algebra $\mfg$ on $(M,\F)$. This transverse action induces the structure of a $\mfg$-dga on $\Omega(M,\F)$ by Proposition~\ref{prop:structureofgdga}. As in the case of the de Rham complexes of complex manifolds, the transverse complex structure yields a bigrading
\begin{equation}\label{eq:bigrading}
\Omega(M,\F) = \bigoplus_{p,q} \Omega^{p,q}(M,\F)
\end{equation}
on $\Omega(M,\F)$ whose total complex is $\Omega^{\bullet}(M,\F) = \bigoplus_{k} \Omega^{k}(M,\F)$.
\begin{prop}\label{prop:eqdol}
The $\mfg$-dga structure on $\Omega^\bullet(M,\F)$ induces a structure of $\mfg$-dga of Dolbeault type on $\Omega^{\bullet,\bullet}(M,\F)$.
\end{prop}
\begin{proof}
We see that the three conditions (1), (2) and (3) in Definition~\ref{defn:holdga} are satisfied as follows:  (1) is satisfied because we have a decomposition $d_{b}=\partial_{b} + \delbar_{b}$ as we saw in~\eqref{eq:doldiff0}. For any $X \in \mfg$ the induced transverse vector field $\trX$ decomposes as $\trX=Z+\overline{Z}$ into its $(1,0)$- and $(0,1)$-components. We get $\iota_{X} = \iota_{Z} + \iota_{\overline{Z}}$ and $\iota_Z$ and $\iota_{\overline{Z}}$ are the $(-1,0)$- and the $(0,-1)$-components of $\iota_{X}$, respectively. Then (2) is satisfied. For any $X\in \mfg$ the operators $L_X$ respect the bidegree, because the $\mfg$-action is holomorphic: (3) is also satisfied.
\end{proof}
Then the equivariant Dolbeault cohomology of $\Omega^{\bullet,\bullet}(M,\F)$ is defined as explained in the last section; we will call it the {\em equivariant basic Dolbeault cohomology} of the $\mfg$-action on $(M,\F)$. More precisely, we have the equivariant $\overline\partial$-operator $\overline\partial_\mfg:=d^{0,1}_\mfg:\Omega^{p,q}_\mfg(M,\F)\to \Omega^{p,q+1}_\mfg(M,\F)$, and define the equivariant basic Dolbeault cohomology by
$$
H^{p,\bullet}_\mfg(M,\F):=H^{\bullet}(\Omega^{p,\bullet}_\mfg(M,\F),\overline\partial_\mfg)\;,
$$
where $\Omega^{p,q}_\mfg(M,\F)=C^{p,q}_\mfg(\Omega^{\bullet,\bullet}(M,\F))$ is the bigraded Cartan complex of the $\mfg$-dga of Dolbeault type $\Omega(M,\F)$. Note that we obtain an $S^{\bullet}(\mfg^*)^\mfg$-algebra $H^\bullet_{\mfg,\delbar}(M,\F)$ by passing to the total grading:
$$
H^k_{\mfg,\delbar}(M,\F):=\bigoplus_{p+q=k} H^{p,q}_\mfg(M,\F)\;.
$$

\begin{rem}
Just as in the case of equivariant basic cohomology, if the foliation in question is the trivial foliation by points, and the $\mfg$-action is induced by an action of a compact Lie group by holomorphic transformations, this notion coincides with ordinary equivariant Dolbeault cohomology as introduced by Teleman~\cite[p.~23]{Teleman} and Lillywhite~\cite[Section~5.1]{Lillywhite}.
\end{rem}

\section{Hodge decomposition for equivariant basic Dolbeault cohomology}\label{sec:eqHodge}

Let $\F$ be a transversely holomorphic Killing foliation on a compact manifold $M$. Completely analogously to the case of torus actions on compact K\"ahler manifolds treated in~\cite[Theorem~5.1]{Lillywhite} one proves
\begin{thm}[Equivariant basic Hodge decomposition] \label{thm:eqbasichodge} Consider a holomorphic transverse action of an abelian Lie algebra $\mfh$ on $(M,\F)$ such that
\begin{equation}\label{eq:ave}
H(\Omega^{\bullet,\bullet}(M,\F)^\mfh,\overline{\partial}_{b})=H^{\bullet,\bullet}(M,\F)\;.
\end{equation}
If the $\mfh$-action is equivariantly formal, then it is also Dolbeault equivariantly formal, i.e.,
\begin{equation}\label{eq:cplexef}
H^{\bullet,\bullet}_\mfh(M,\F)\cong S^{\bullet}(\mfh^*)\otimes H^{\bullet,\bullet}(M,\F)
\end{equation}
as a bigraded $S^{\bullet}(\mfh^*)$-module, and there is a graded $S^{\bullet}(\mfh^*)$-module isomorphism
\begin{equation}\label{eq:eqHodge}
H^\bullet_{\mfh}(M,\F)\cong H^\bullet_{\mfh,\delbar}(M,\F)\;.
\end{equation}
\end{thm}

\begin{rem}
For the Hodge decomposition for equivariant Dolbeault cohomology proven by Lillywhite~\cite{Lillywhite}, no assumption like~\eqref{eq:ave} on the cohomology of the complex of invariant forms is needed. This is because he considers actions of compact Lie groups for which an averaging process implies~\eqref{eq:ave} (see Corollary~\ref{cor:eqfor} below).
\end{rem}

\begin{proof}[Proof of Theorem~\ref{thm:eqbasichodge}]
The equivariant basic differential $d_\mfh:\Omega^\bullet_\mfh(M,\F)\to \Omega^\bullet_\mfh(M,\F)$ decomposes as $d_\mfh=\overline{\partial}_\mfh+ \partial_\mfh$. We obtain an equivariant Fr\"olicher spectral sequence with $E_1$-term equal to the equivariant basic Dolbeault cohomology, and converging to the equivariant basic cohomology $H^\bullet_\mfh(M,\F)$, which because of the assumption of equivariant formality is isomorphic to $S^{\bullet}(\mfh^*)\otimes H^\bullet(M,\F)$.

In turn, the $E_1$-term is the limit of the spectral sequence of the double complex obtained by the decomposition $\overline{\partial}_\mfh=\overline{\partial}_{b} + \delta$, where $(\delta\omega)(X)=-\iota_{X}^{-1,0} (\omega(X))$. Note that $\Omega^{\bullet,\bullet}_\mfh(M,\F) = S^{\bullet}(\mfh^*)\otimes \Omega^{\bullet,\bullet}(M,\F)^{\mfh}$ because $\mfh$ is abelian. By~\eqref{eq:ave}, this spectral sequence has $E_1$-term
\begin{align*}
H(S^{\bullet}(\mfh^*)\otimes \Omega^{\bullet,\bullet}(M,\F)^\mfh,\overline{\partial}_{b}) &\cong S^{\bullet}(\mfh^*)\otimes H(\Omega^{\bullet,\bullet}(M,\F)^\mfh,\overline{\partial}_{b})\\
&\cong S^{\bullet}(\mfh^*)\otimes H^{\bullet,\bullet}(M,\F)\;.
\end{align*}
Killing foliations are homologically orientable by~\cite[Th\'eor\`eme~I]{Sergiescu}. Passing to the total grading, the basic Hodge decomposition (Theorem~\ref{thm:hodgedecomp}) implies that this $E_1$-term already coincides with the limit of the equivariant Fr\"olicher spectral sequence above. Hence, both spectral sequences collapse at the $E_1$-term, which implies~\eqref{eq:cplexef} and~\eqref{eq:eqHodge}.
\end{proof}

We will see that the assumption~\eqref{eq:ave} is satisfied in some important cases. The proof of the following Lemma is the basic version of~\cite{MO}.
\begin{lemma}\label{lem:trivialactionondolbeault}
Let $\mfh$ be a Lie algebra acting transversely holomorphically on $(M,\F)$. Then $\mfh$ acts trivially on basic Dolbeault cohomology: $H^{\bullet,\bullet}(M,\F)^\mfh=H^{\bullet,\bullet}(M,\F)$.
\end{lemma}
\begin{proof}
Killing foliations are homologically orientable by~\cite[Th\'eor\`eme~I]{Sergiescu}. So, by the basic Hodge decomposition described in Section \ref{sec:HodgetrKaehler}, any basic Dolbeault cohomology class is represented by a form $\omega\in \Omega^{p,q}(M,\F)$ that is closed both with respect to $d$ and $\delbar_b$. As the $\mfh$-action is transversely holomorphic, $L_{\widetilde{X}}\omega=d\iota_{\widetilde{X}}\omega$ is again a basic $(p,q)$-form. Thus, decomposing ${\widetilde{X}}=Z+\overline{Z}$ into the $(1,0)$- and $(0,1)$-components, we get
\begin{equation}\label{eq:mof}
L_{\widetilde{X}}\omega = \delbar_b(\iota_{\overline{Z}}\omega) + \partial_b(\iota_Z\omega)
\end{equation}
and the other two summands $\partial_b(\iota_{\overline{Z}}\omega)$ and $\delbar_b(\iota_Z\omega)$ vanish for degree reasons. As $\delbar_b(\iota_Z\omega)=0$, we find a basic $(p-1,q-1)$-form $\eta$ such that $\iota_Z\omega + \delbar_b \eta$ is $d$-closed. Thus,
\[
\partial_b(\iota_Z\omega) = -\partial_b\delbar_b \eta = \delbar_b  \partial_b\eta\;,
\]
and plugging that into \eqref{eq:mof} we get
\[
L_{\widetilde{X}}\omega = \delbar_b(\iota_{\overline{Z}}\omega + \partial_b\eta)\;,
\]
which shows that the induced action on basic Dolbeault cohomology is trivial.
\end{proof}

\begin{cor}\label{cor:eqfor}
Assume that $\mfh$ is equal to the structural Killing algebra $\mfa$ of $(M,\F)$ or the $\mfh$-action is the infinitesimal action of a holomorphic torus action. Then the equivariant formality of the $\mfh$-action implies the Dolbeault equivariantly formality and~\eqref{eq:eqHodge}.
\end{cor}

\begin{proof}
By Lemma~\ref{lem:trivialactionondolbeault}, we see that our condition $H(\Omega^{\bullet,\bullet}(M,\F)^\mfh,\overline{\partial}_{b})=H^{\bullet,\bullet}(M,\F)$ of Theorem~\ref{thm:eqbasichodge} is always satisfied in case the $\mfh$-action is induced by the action of a torus. Also it is satisfied if $\mfh=\mfa$ is the structural Killing algebra of $\F$, as in this case we even have $\Omega^{\bullet,\bullet}(M,\F)^\mfa=\Omega^{\bullet,\bullet}(M,\F)$, see~\cite[Lemma~3.15]{GT2010}. Thus Corollary~\ref{cor:eqfor} follows from Theorem~\ref{thm:eqbasichodge}.
\end{proof}

As a corollary, we obtain a basic version of a result of Carrell and Lieberman~\cite[Theorem~1]{CarrellLieberman}. Note that a proof of the result of Carrell and Lieberman using equivariant Dolbeault cohomology is given in~\cite[Theorem~4.6]{CKP}. Our proof is same as the proof in~\cite{CKP} adapted to the basic situation. Let $C$ be the union of closed leaves of $\F$, and note that $C/\F$ naturally admits the structure of a K\"ahler orbifold.
\begin{thm}\label{thm:carlieb}
Let $\F$ be a transversely K\"ahler Killing foliation on a compact manifold $M$ such that the transverse action of the structural Killing algebra $\mfa$ is equivariantly formal. Then
$$
\sum_p h^{p,p+s}(M,\F)=\sum_p h^{p,p+s}(C,\F)=\sum_p h^{p,p+s}(C/\F)
$$
for all $s$. In particular, for $|s|>\dim_\CC C/\F$, we have $h^{p,p+s}(M,\F)=0$ and $H^{p,p+s}_\mfa(M,\F)=0$.
\end{thm}
\begin{proof}
Because the $\mfa$-action is Dolbeault equivariantly formal by Corollary~\ref{cor:eqfor}, the natural inclusion $C\to M$ induces an injective map
$$
H^{\bullet,\bullet}_\mfa(M,\F) \to H^{\bullet,\bullet}_\mfa(C,\F)\;.
$$
Its cokernel is torsion due to~\cite[Theorem~5.2]{GT2010} and the equivariant basic Hodge decomposition. The diagonals $\bigoplus_p H^{p,p+s}$ of the left and right hand side are $S^{\bullet}(\mfa^*)$-submodules and the cokernels of the map restricted to the diagonals are again torsion, hence the rank of the corresponding diagonals as $S^{\bullet}(\mfa^*)$-modules are equal. By Equation~\eqref{eq:cplexef} these are $\sum_p h^{p,p+s}(M,\F)$ and $\sum_p h^{p,p+s}(C,\F)$, respectively. $h^{p,p+s}(C,\F)=h^{p,p+s}(C/\F)$ is a direct consequence of the definition of basic Dolbeault cohomology.

If $|s|>\dim_\CC C/\F$, then $h^{p,p+s}(C,\F)=0$ for all $p$. Thus, $h^{p,p+s}(M,\F)=0$. By~\eqref{eq:cplexef}, we have also $H^{p,p+s}_\mfa(M,\F)=0$ for all $p$.
\end{proof}

\section{A vanishing theorem for Sasakian manifolds}\label{sec:vanSas}
On a compact Sasakian manifold $M$, the characteristic foliation $\F$ is transversely K\"ahler and admits a natural transverse action of an Abelian Lie algebra $\mfa$ as mentioned in Section~\ref{sec:tractiononsas}. $\F$ is Killing and $\mfa$ is regarded as the structural Killing algebra of $\F$ as mentioned in Section~\ref{sec:tractiononKilling}. Since this $\mfa$-action is always equivariantly formal by~\cite[Theorem~6.8]{GNT}, Theorems~\ref{thm:eqbasichodge} and~\ref{thm:carlieb} apply. For example:
\begin{thm}\label{thm:vanishing}
If the characteristic foliation $\F$ of a Sasakian structure on a compact manifold $M$ has only finitely many closed leaves, then $h^{p,q}(M,\F)=0$ for $p\neq q$.
\end{thm}
As a class of examples where this theorem applies, recall the following definition:
\begin{defn}\label{defn:toriccontact}
A $(2n+1)$-dimensional contact manifold with a $T^{n+1}$-action preserving the contact structure is called a {\em contact toric manifold}. Moreover, if the Reeb vector field of a contact form generates an $\RR$-subaction of the $T^{n+1}$-action, then the contact $T^{n+1}$-manifold is called a {\em contact toric manifold of Reeb type}.
 \end{defn}

\begin{thm}[\cite{Boyer Galicki 2}]\label{thm:toriccontact}
A contact toric manifold of Reeb type admits a Sasakian structure.
\end{thm}
Thus as a corollary of Theorem~\ref{thm:vanishing}, we get
\begin{cor}\label{cor:vantor}
For an invariant Sasakian structure on a contact toric manifold $M$ of Reeb type with characteristic foliation $\F$, we have $h^{p,q}(M,\F)=0$ for $p\neq q$.
\end{cor}
Finally, we mention that Theorems~\ref{thm:invtypeI} and~\ref{thm:vanishing} can be combined to the following statement; recall that a CR vector field by definition is a vector field whose flow preserves the CR structure.

\begin{thm}\label{thm:rigidityandvanishing}
Let $(M,\eta,g)$ be a compact Sasakian manifold with characteristic foliation $\F$. If there exists a nowhere vanishing CR vector field $X$ on $M$ with only finitely many closed orbits, then $h^{p,q}(M,\F)=0$ for $p\neq q$.
\end{thm}

\begin{proof}
By Proposition~\ref{prop:sp}, it suffices to show the case where $M$ is not diffeomorphic to $S^{2n+1}$. Then the CR diffeomorphism group $\CR(\cD,J)$ is compact by a theorem of Schoen~\cite{Schoen}. By Lemma~\ref{lem:cone}, $(\eta,g)$ can be obtained from a $\CR(\cD,J)$-invariant Sasakian structure $(\eta_{0},g_{0})$ under a deformation of type I. Let $\xi_{0}$ be the Reeb vector field of $(\eta_{0},g_{0})$ and $\F_{0}$ the characteristic foliation of $(\eta_{0},g_{0})$. Here the flow generated by $\xi_{0}$ commutes with the $\CR(\cD,J)$-action. In particular, the flow generated by $\xi_{0}$ maps each closed $X$-orbit to another closed $X$-orbit. Then, by the finiteness of the closed $X$-orbits, each closed $X$-orbit is preserved by the flow generated by $\xi_{0}$. Hence the torus subgroup $S$ generated by $X$ and $\xi_{0}$ has only finitely many $1$-dimensional orbits. Then, by a deformation of type I using a generic infinitesimal generator of the $S$-action, we get a new Sasakian structure $(\eta_{1},g_{1})$ such that the closed leaves of the characteristic foliation $\F_{1}$ are equal to the $1$-dimensional orbits of the $S$-action. By Theorem~\ref{thm:vanishing}, we get $h^{p,q}(M,\F_{1})=0$ for $p\neq q$. Theorem~\ref{thm:defI} implies $h^{p,q}(M,\F)=h^{p,q}(M,\F_{0})=h^{p,q}(M,\F_{1})$, which concludes the proof.
\end{proof}

\begin{ex}
Let us consider the well-understood case of a three-dimensional Sasakian manifold $(M,\eta,g)$~\cite{Geiges,Belgun}, see also~\cite[Section~7]{BGM}. If $M$ is null (resp., negative), then it is, up to covering, a circle bundle over a complex torus (resp., a Riemann surface of genus at least $2$). In both cases, the only possible deformations of type I are rescalings of the Reeb vector field. In particular, we cannot find a Sasakian structure with the same underlying CR structure whose characteristic foliation has only finitely many closed leaves. This fact is reflected in cohomology, as in both cases the off-diagonals in basic Dolbeault cohomology do not vanish. Note that any irregular Sasakian $3$-manifold is toric. By~\cite[Theorem~2.18]{Lerman}, such manifolds are lens spaces.
\end{ex}

\begin{ex} More generally, consider a regular Sasakian structure on the total space of a circle bundle over a K\"ahler manifold $M$ obtained by the Boothby-Wang construction~\cite{Boothby Wang}. In case the off-diagonals in the Dolbeault cohomology of $M$ do not vanish, the Sasakian structure does not admit deformations of type I such that the characteristic foliation of the deformed Sasakian structure has only finitely many closed leaves.
\end{ex}

\section{Deformations of homogeneous Sasakian manifolds}\label{sec:G/P}

As an application of the results in this paper, we calculate the basic Hodge numbers of Sasakian structures whose characteristic foliation has finitely many closed leaves constructed by deforming homogeneous Sasakian manifolds. It is well-known that any compact homogeneous Sasakian manifold $(M,\eta,g)$ is a nontrivial circle bundle over a generalized flag manifold, see~\cite[Theorem~8.3.6]{Boyer Galicki}. Denote the underlying CR structure of $(M,\eta,g)$ by $(\cD,J)$, and the Reeb vector field by $\xi$. We fix a compact Lie group $G\subset \AUT(\eta,g)$ that contains the Reeb flow of $\eta$ as a one-parameter subgroup and acts transitively on $M$. Then we can write this circle bundle as $M=G/K\to G/H$, where $H$ is the centralizer of a torus in $G$ (in particular, $\rk G = \rk H$), and $G/H$ is a homogeneous K\"ahler manifold. Note that $\xi$ is contained in the center of $\mfg$ and acts trivially on $G/H$, i.e., is also contained in $\mfh$.

\begin{thm}\label{thm:Madmits}
$M$ admits an irregular Sasakian structure such that the characteristic foliation $\F$ has a finite number of closed leaves and
\begin{equation}\label{eq:homhpq}
h^{p,q}(M,\F) = h^{p,q}(G/H)=
\begin{cases}
b^{2k}(G/H) & \text{if}\,\,\,\, p=q=k \;, \\
0 & \text{if}\,\,\,\, p \neq q\;.
\end{cases}
\end{equation}
The number of closed leaves of $\F$ is $\chi(G/H)$, the Euler number of $G/H$.
\end{thm}

\begin{rem}
Originally, $h^{p,q}(G/H)=0$ for $p\neq q$, the second equality in~\eqref{eq:homhpq}, was shown by Borel-Hirzebruch~\cite[Proposition in~14.10]{BH0}. It also follows from the original vanishing theorem of Carrell and Lieberman~\cite{CarrellLieberman}. Here we will deduce it from Theorems~\ref{thm:invtypeI} and \ref{thm:vanishing}.
\end{rem}

\begin{proof}[Proof of Theorem~\ref{thm:Madmits}]
Each element in the open convex cone $\mathfrak{cr}^+(\cD,J)$ (see Section~\ref{sec:invarianceofhpq}) is the Reeb vector field of another Sasakian structure on $M$ with CR structure $(\cD,J)$. Because $\xi\in \mfh\cap \mathfrak{cr}^+(\cD,J)$, a small neighborhood of $\xi$ in $\mfg$ is contained in $\mathfrak{cr}^+(\cD,J)$. The closed leaves of the characteristic foliation of the Sasakian structure corresponding to an element in $\mfg\cap \mathfrak{cr}^+(\cD,J)$ are exactly the preimages under the projection $G/K\to G/H$ of the fixed points of the flow on $G/H$ generated by this element. We therefore need to find an element in $\mfg\cap \mathfrak{cr}^+(\cD,J)$ whose flow has only finitely many fixed points on $G/H$.

Let $T$ be a maximal torus in $H$ (which then is also a maximal torus in $G$). The fixed point set of the action of $T$ by left multiplication on $G/H$ is exactly $N_G(T)/N_H(T)$, which (because $\rk G = \rk H$) is equal to the quotient of Weyl groups $W(G)/W(H)$, so in particular finite. This means that the flow of a generic element in $\mft$ close to $\xi$ has only finitely many fixed points on $G/H$.

Theorem~\ref{thm:vanishing} implies that the associated Sasakian structure satisfies $h^{p,q}(M,\F)=0$ for $p\neq q$. By Theorem~\ref{thm:invtypeI}, it has the same basic Hodge numbers as the regular Sasakian structure on $M$ we started with, but the basic Hodge numbers of that Sasakian structure are the same as the Hodge numbers of $G/H$. Since $h^{p,q}(G/H)=h^{p,q}(M,\F)=0$ for $p \neq q$, we get
$$
h^{k,k}(M,\F)=\sum_{p+q=2k} h^{p,q}(M,\F) = \sum_{p+q=2k} h^{p,q}(G/H)=b^{2k}(G/H)\;.
$$
By~\cite[Theorem 7.11]{GNT}, the number of closed leaves of the characteristic foliation of the deformed Sasakian structure equals the total basic Betti number $\sum_{k} b^{k}(M,\F)= \sum_{k} b^{k} (G/H)$. Since $H^{\odd}(G/H)=0$, we get $\sum_{k} b^{k} (G/H) = \chi (G/H)$.
\end{proof}

If we now assume that $G$ and $H$ are connected, the Betti numbers of $G/H$ have been calculated by Borel \cite[Theorem~26.1 (c)]{Borel} in terms of the Betti numbers of $G$ and $H$: Let
\begin{align*}
P_{t}(G) & = \prod_{i=1}^{r}(1 + t^{g_{i}})\;, &  P_{t}(H) & = \prod_{i=1}^{r}(1 + t^{l_{i}})
\end{align*}
be the Poincar\'{e} polynomials of $G$ and $H$, where $r=\rank G=\rank H$. Then the Poincar\'{e} polynomial $P_{t}(G/H)$ of $G/H$ is given by
\begin{equation}\label{eq:Pt}
P_{t}(G/H) = \prod_{i=1}^{r}\frac{1 - t^{g_{i}+1}}{1-t^{l_{i}+1}}\;,
\end{equation}
which implies that the Euler number $\chi(G/H)$ of $G/H$ is
\begin{equation}\label{eq:Eul}
\chi(G/H) = \prod_{i=1}^{r}\frac{g_{i}+1}{l_{i}+1}\;.
\end{equation}
By these formulas, we get the following examples.

\begin{ex}
For the case of $G=\SU(n)$ and $H=\S(\U(p) \times \U(n-p))$, we have
\begin{align*}
P_{t}(G) & = \prod_{i \in \{3,5, \ldots, 2n-1\}}(1 + t^{i})\;, & P_{t}(H) & = \prod_{j }(1 + t^{j})\;,
\end{align*}
where $j$ runs in $\{1, 3, 5, \ldots, 2(n-p)-1, 3, 5, \ldots, 2p-1\}$. Then, by Theorem~\ref{thm:Madmits} and Equation~\eqref{eq:Pt}, we get an irregular Sasakian structure on a circle bundle $M$ over $G/H$ whose basic Hodge numbers are determined by the basic Poincar\'e polynomial
\[
P_{t}(M,\F) = \prod_{i=1}^{n-p}\frac{1 - t^{2i+2}}{1-t^{2i}} \cdot \prod_{i=1}^{p-1}\frac{1 - t^{2(n-p+1)+2i}}{1-t^{2i+2}}\;
\]
and, by~\eqref{eq:Eul}, whose number of closed leaves of $\F$ equals $\chi(G/H) = \frac{n!}{(n-p)!p!}$.
\end{ex}

\begin{ex}
For the case of $G=E_{7}$ and $H=E_{6}\cdot \SO(2)$, we have
\begin{align*}
P_{t}(G) & = \prod_{i}(1 + t^{i})\;, & P_{t}(H) & = \prod_{j}(1 + t^{j})\;,
\end{align*}
where $i$ runs in $\{3, 11, 15, 19, 23, 27, 35\}$ and $j$ runs in $\{1, 3, 9, 11, 15, 17, 23\}$. Then, by Theorem~\ref{thm:Madmits} and Equation~\eqref{eq:Pt}, we get an irregular Sasakian structure on a circle bundle $M$ over $G/H$ whose basic Hodge numbers are determined by the basic Poincar\'e polynomial
\[
P_{t}(M,\F) = \frac{\prod_{i \in \{3, 11, 15, 19, 23, 27, 35\}} (1 - t^{i+1})}{\prod_{j \in \{1, 3, 9, 11, 15, 17, 23\}} (1 - t^{j+1})}
\]
and, by~\eqref{eq:Eul}, whose number of closed leaves of $\F$ equals $\chi(G/H) = 56$.
\end{ex}

\end{document}